\documentclass[letterpaper,10pt,reqno,onefignum,onetabnum]{amsart}
\usepackage[english]{babel}
\usepackage{amsmath}
\usepackage{amsthm}
\usepackage{verbatim}
\usepackage{mathrsfs}
\usepackage{bm}
\usepackage[foot]{amsaddr}
\usepackage[dvipsnames]{xcolor}
\usepackage{hyperref}
\hypersetup{
	colorlinks = true,
	linkcolor = OliveGreen,
	anchorcolor = OliveGreen,
	citecolor = OliveGreen,
	filecolor = OliveGreen,
	urlcolor = OliveGreen
}
\usepackage{algorithm}% http://ctan.org/pkg/algorithms
\usepackage{algorithmic}% http://ctan.org/pkg/algorithms
\usepackage{float}
\usepackage{lipsum}
\usepackage{amsfonts}
\usepackage{amssymb}
\usepackage{graphicx}
\usepackage{epstopdf}
\usepackage{cases}
\usepackage{multirow}
\ifpdf
\DeclareGraphicsExtensions{.eps,.pdf,.png,.jpg}
\else
\DeclareGraphicsExtensions{.eps}
\fi
\usepackage{verbatim}
\usepackage{mathrsfs}
\usepackage{bm}

\newtheorem{teo}{Theorem}[section]
\newtheorem{prop}[teo]{Proposition}

\newtheorem{cor}[teo]{Corollary}
\newtheorem{pro}[teo]{Problem}
\newtheorem{algo}[teo]{Algorithm}

\newtheorem{rem}[teo]{Remark}

\usepackage{lipsum}
\usepackage{amsfonts}
\usepackage{amssymb}
\usepackage{graphicx}
\usepackage{epstopdf}
\usepackage{cases}
\usepackage{multirow}
\usepackage{algorithmic}
\usepackage{tikz}
\usepackage{booktabs}%
\usepackage{hhline}
\usetikzlibrary{matrix}
\usetikzlibrary{arrows}
\ifpdf
\DeclareGraphicsExtensions{.eps,.pdf,.png,.jpg}
\else
\DeclareGraphicsExtensions{.eps}
\fi
\usepackage{verbatim}
\usepackage{mathrsfs}
\usepackage{bm}
\usepackage{color}
\usepackage{xcolor}
\usepackage{caption}
\usepackage{subcaption}
\usepackage{enumitem}
\newcommand{\N}{\mathbb N}

\newcommand{\R}{\mathbb R}

\renewcommand{\H}{\mathcal{H}}
\newcommand{\G}{\mathcal G}

\newcommand{\id}{\textnormal{Id}}
\newcommand{\x}{\bm x}

\newcommand{\weak}{\rightharpoonup}
\newcommand{\ran}{\textnormal{ran}\,}
\newcommand{\dom}{\textnormal{dom}\,}

\newcommand{\zer}{\textnormal{zer}}

\newcommand{\gra}{\textnormal{gra}\,}

\newcommand{\scal}[2]{{\left\langle{{#1}\mid{#2}}\right\rangle}}

\newcommand{\menge}[2]{\big\{{#1}~\big |~{#2}\big\}}

\newcommand{\RPP}{\ensuremath{\left]0,+\infty\right[}}

\newcommand{\RX}{\ensuremath{\left]-\infty,+\infty\right]}}

\newcommand{\sri}{\ensuremath{\text{\rm sri}\,}}

\newcommand{\prox}{\ensuremath{\text{\rm prox}\,}}

\newcommand{\weakly}{\ensuremath{\:\rightharpoonup\:}}

\usepackage{geometry}
\numberwithin{equation}{section}

\numberwithin{equation}{section}

\DeclareFontEncoding{FMS}{}{}
\DeclareFontSubstitution{FMS}{futm}{m}{n}
\DeclareFontEncoding{FMX}{}{}
\DeclareFontSubstitution{FMX}{futm}{m}{n}
\DeclareSymbolFont{fouriersymbols}{FMS}{futm}{m}{n}
\DeclareSymbolFont{fourierlargesymbols}{FMX}{futm}{m}{n}
\DeclareMathDelimiter{\nr}{\mathord}{fouriersymbols}{152}{fourierlargesymbols}{147}

\DeclareMathDelimiter{\nr}{\mathord}{fouriersymbols}{152}{fourierlargesymbols}{147}
\DeclareMathAlphabet{\mathpzc}{OT1}{pzc}{m}{it}

\title[Forward-Primal-Dual-Half-forward 
algorithm]{Forward Primal-Dual Half-Forward Algorithm for Splitting Four Operators}

\author{Fernando Rold\'an$^{\dagger}$}
\address{$^{\dagger}$ Departamento de Ingeniería Matemática, Universidad de Concepción, Concepción, Chile.}
\email{fernandoroldan@udec.cl}

\begin{document}
	\begin{abstract}
		In this article, we propose a splitting algorithm to find zeros of the sum of four maximally monotone operators in real Hilbert spaces. In particular, we consider a Lipschitzian operator, a cocoercive operator, and a linear composite term. In the case when the Lipschitzian operator is absent, our method reduces to the Condat-V\~u algorithm. On the other hand, when the linear composite term is absent, the algorithm reduces to the Forward-Backward-Half-Forward algorithm (FBHF). Additionally, in each case, the set of step-sizes that guarantee the weak convergence of those methods are recovered. Therefore, our algorithm can be seen as a generalization of Condat-V\~u and FBHF. Moreover, we propose extensions and applications of our method in multivariate monotone inclusions and saddle point problems. Finally, we present a numerical experiment in image deblurring problems.		
		\par
		\bigskip
		
		\noindent \textbf{Keywords.} {\it Operator splitting, monotone operators, monotone inclusion, convergence analysis, convex 			optimization}
		\par
		\bigskip \noindent
		2020 {\it Mathematics Subject Classification.} {47H05, 65K05 
			65K15, 90C25.}
		%62H35, 94A08,
		
	\end{abstract}
	
	\maketitle
	
\section{Introduction}

The problem of finding zeros of monotone operators has significant applications in fields such as convex 
programming \cite{Combettes2018MP}, 
game theory \cite{Nash13}, data science \cite{CombettesPesquet2021strategies}, image recovery 
\cite{BotHendrich2014TV,Briceno2011ImRe}, among others. In this article, we propose an 
algorithm for finding zeros of $A+L^*BL+C+D$, where $A$ and $B$ are maximally monotone 
operators, $L$ is a bounded 
linear 	operator and $L^*$ denotes its adjoint, $C$ is a Lipschitz operator, and $D$ is a 
cocoercive operator. Below we provide an overview of existing methods capable of solving this problem or specific instances of it.\\
In the case when $C=0$ this problem 
can be solve by the Condat-V\~u algorithm (CV) \cite{Condat13,Vu13}. This method is a primal 
dual algorithm which splits and leverages all the specific properties of the operators involved. CV consists of two backward steps by the resolvent of $A$ and $B$, one forward step by {evaluating} the operator $C$, and one 
{  evaluation} of $L$ and other of $L^*$. In the case when $D=0$, CV reduces to the 
primal-dual algorithm (PD), also known as the Chambolle-Pock algorithm \cite{ChambollePock2011}. Moreover, the relationship between PD and the proximal point algorithm \cite{martinet1970,RockafellarSIAM1976} was established in \cite{HeYuan2012}.  
Furthermore, in scenarios where $L=\id$, PD can be viewed as the Douglas-Rachford splitting (DR) \cite{Eckstein1992,Lions1979SIAM,Svaiter2011} by considering an auxiliary sequence (see \cite{BricenoRoldanrange,BricenoRoldan2021,OConnor2020}). An additional method solving 
the case $C=D=0$ can be found in \cite{briceno2011SIAM}. In the case when $L=\id$, 
the problem can be addressed by the Three-Operator-Splitting proposed in \cite{davis2015}. This method is a generalization of DR incorporating an {  evaluation} of $D$ in its iterations. \\
On the other hand, when $B=0$ and $L=0$, the problem can be addressed using 
the Forward-Backward-Half-Forward 
algorithm (FBHF) \cite{BricenoDavis2018}. This method splits all the operators
and consists in a double forward step by {  evaluating} $C$ and $D$, a backward step by calculating 
the resolvent of $A$, and a final forward step by {  evaluating} $C$. Additionally, by using 
product spaces techniques, the authors in \cite{BricenoDavis2018} proposed an algorithm for 
solving our problem when $B\neq 0$ and $L\neq 0$. If $C=0$ FBHF reduces 
to the Forward-Backward algorithm \cite{Lions1979SIAM}. Moreover, when $D=0$, FBHF reduces to 
the Tseng's splitting or Forward-Backward-Forward algorithm (FBF) \cite{Tseng2000SIAM}. Other 
methods dealing with Lipschitz operators can be found in 
\cite{Cevher2020SVVA,Csetnek2019AMO,Malitsky2020SIAMJO}. These methods solve our problem when 
$D=0$ and require only one {  evaluation} of the Lipschitz 
operator. Particularly, the authors in \cite{Malitsky2020SIAMJO} proposed the method called 
Forward-Reflected-Backward (FRB). Furthermore, they proposed an extension considering 
cocoercive operators called Forward-Half-Reflected-Backward (FHRB). 
Recently, in \cite{MorinBanertGiselsson2022,Giselsson2021NFBS} the authors proposed Forward-Backward 
type algorithms with non-linear warped resolvents (see 
\cite{BuiCombettesWarped2020,Bredies2022DPP}) through which, by selecting specific
non-linear Lipschitz operators, it is possible to establish the convergence of FBHF and FHRB. Moreover, by using a suitable non-linear warped resolvent in a primal-dual setting, the authors in \cite{MorinBanertGiselsson2022} proposed two methods to solve our problem. \\
When $D=0$ the problem can be solved by the method proposed in \cite{Combettes2012SVVA}. In this article, the 
authors consider a broader context involving a mixture of compositions and parallel sums. This 
method is derived by implementing FBF in a primal-dual setting.\\
The methods proposed in \cite{Rieger2020AMC,RyuVu2020} can solve our 
problem when $L=\id$ and $D=0$. These algorithms can be seen as a combination of DR with FRB and DR
with the recurrence proposed in \cite{Cevher2020SVVA}. These 
methods has been extended for addressing the case $D\neq 0$ in \cite{Tang2022fouroperator}.\\
Note that, since $C+D$ is a Lipschitz operator, the methods dealing with Lipschitz operators 
can address the case $D\neq 0$ by iterating $C+D$ as a Lipschitz operator. 
\\
In this article we propose an algorithm for solving the monotone inclusion which fully split the 
structure of the problem. This method can be seen as a combination of Condat-V\~u and FBHF: in the 
case when $C=0$ our method reduce to Condat-V\~u and when $B=0$ and $L=0$ our method reduces to 
FBHF. Additionally, in each case, we recover the convergence conditions on the step-sizes. In the 
particular case when $D=0$, our method combines PD and FBF, which is also a new contribution. We 
also proposed extension and applications of our method for solving multivariate monotone inclusions and saddle point problems. Finally, we present numerical 
experiments in image restoration, comparing the performance of the proposed algorithm with the methods proposed in \cite{BricenoDavis2018,MorinBanertGiselsson2022}. {  In particular, the proposed algorithm allows for larger step-sizes, which generates numerical advantages.}
\\
The article is organized as follows. In Section~\ref{sec:Problem} we present the problem to be
solved and the main algorithm. In Section~\ref{sec:Pre} we state our notations and preliminaries results. In Section~\ref{sec:conv} we analyze the weak convergence of our method, 
providing our main results and corollaries. In Section~\ref{sec:ext} we provide an extension of our method for solving multivariate monotone inclusions and  applications to saddle point problems. Finally, we dedicate Section~\ref{sec:exp} to 
numerical experiments.

\section{Problem Statement and Proposed Algorithm}\label{sec:Problem}
In this article, we aim to solve numerically the following problem.
\begin{pro}\label{pro:main}
	Let $(\H,\scal{\cdot}{\cdot})$ and $(\G,\scal{\cdot}{\cdot})$ 
	be real Hilbert spaces, let $A: \H \to 2^\H$ be a maximally 
	$\rho$-monotone operator for 
	some $\rho\in \R$, let $B:\G \to 
	2^\G$ 
	be a maximally monotone operator, let $L\colon \H \to \G$ be 
	a linear bounded operator with adjoint $L^*\colon \G \to \H$, 
	let 
	$C:\H 
	\to \H$ be a $\zeta$-Lipschitz operator for $\zeta \in 
	\RPP$, and let 
	$D:\H \to \H$ be a 
	$\beta$-cocoercive operator for $\beta \in \RPP$. Suppose 
	that 
	$A+C$ is maximally monotone. The problem is to
	\begin{equation} 
		\text{find }  (x,u) \in \H \times \G \text{ such that } 
		\begin{cases}
			0 &\in (A+C+D)x+L^*u\\
			0 &\in B^{-1}u-Lx.
		\end{cases} 
	\end{equation}
	under the hypothesis that its solution set, denoted by 
	$\bm{Z}$,  is not 
	empty.
\end{pro}
Note that, given $(\hat{x},\hat{u}) \in \bm{Z}$, $\hat{x}$ is a 
solution to the primal monotone inclusion
\begin{equation}\label{eq:primalinclu}
	\text{find }  x \in \H \text{ such that } 
	0 \in (A+L^*BL+C+D)x
\end{equation}
and $\hat{u}$ is a solution to the dual monotone inclusion
\begin{equation}\label{eq:dualinclu}
	\text{find }  u \in \G \text{ such that } 
	0 \in  B^{-1}u - L(A+C+D)^{-1}L^*u.
\end{equation}
An specific scenario of Problem~\ref{pro:main} is the following 
optimization problem.
\begin{pro}\label{pro:opti}
	Let $f\colon \H \to \RX$ and $g\colon \G \to \RX$, be proper 
	lower-semicontinuous convex functions, let $h\colon \H \to \R$ be 
	a differentiable 
	convex function with a $\zeta$-Lipschitz 
	gradient for some $\zeta\in \RPP$, and let 
	$d\colon \G \to \R$ be a differentiable convex function
	with a $1/\beta$-Lischitzian gradient for some 
	$\beta \in \RPP$. The problem is to
	\begin{equation}
		\min_{x \in \H} f(x)+g(Lx)+h(x)+d(x).
	\end{equation}
	under the assumption that its solution set is nonempty.
\end{pro}
Under standard qualification conditions, {  for example, $0 \in \sri(\dom g +L(\dom f))$ (see also  \cite[Proposition 6.19]{bauschkebook2017})}, Problem~\ref{pro:opti} is 
equivalent to
\begin{equation}\label{eq:inclupartial}
	\text{find } x \in \H \text{ such that } 0 \in (\partial f+ 
	L^*\partial g L+\nabla h +\nabla d)x
\end{equation}
where $\partial f$ denotes the 
subdifferential of the function $f$. By setting $A = 
\partial 
f$, $B=\partial g$, $C=\nabla h$, and $D = \nabla d$,  
Problem~\ref{pro:opti} can be seen as a particular case of the 
primal monotone inclusion on \eqref{eq:primalinclu}. Note that, by 
\cite[Corollary~18.17]{bauschkebook2017}, 
$\nabla d$ is $\beta$-cocoercive and $\nabla h$ is 
$1/\zeta$-cocoercive, then \eqref{eq:inclupartial} can be solved by the 
Condat-V\~u algorithm by considering the $(\beta+1/\zeta)$-cocoercive 
operator $(\nabla h +\nabla d)$. However, not every 
Lipschitizan operator is cocoercive, {  for instance, consider the linear operator $T\colon \R^2 \to \R^2$ given by $T(x,y)=(-y,x)$. Moreover, even in monotone inclusions associated with optimization problems it might be necessary to deal with non cocoercive Lipschitizian operators (see Section~\ref{section:saddle})}. For instance, in adjoint 
mismatch inverse problems, the solution is approximated by solving a 
relaxed monotone inclusions associated to an optimization 
problem (see \cite{ChouzenouxPesquetRoldan2023}). This leads to a disconnection between the Lipschitz property and cocoercivity. Additional examples of non cocoercive Lipschitz operators arising from 
optimization problems are available in \cite[Example~1]{CombettesMinh2022} and in \cite{Rieger2020AMC}.

In this article, we propose the following algorithm for solving 
Problem~\ref{pro:main}.
\begin{algo}\label{algo:algo1}
	In the context of Problem~\ref{pro:main}, let $(x_0,u_0)\in 
	\H\times \G$, let $(\sigma,\tau)\in \RPP^2$, suppose that 
	$\tau\rho>-1$, and consider the 
	iteration:
	\begin{equation}\label{eq:algo1}
		(\forall n\in\N)\quad 
		\begin{array}{l}
			\left\lfloor
			\begin{array}{l}
				p_{n+1} = C x_n\\
				z_{n+1} = J_{\tau A} (x_n-\tau (L^* u_n+p_{n+1}+D 
				x_n) 
				) \\
				q_{n+1} = \tau(Cz_{n+1}-p_{n+1})\\
				u_{n+1} =  J_{\sigma
					B^{-1}}\left(u_{n}+\sigma L( 
				2z_{n+1}-x_{n}-q_{n+1})\right)\\
				x_{n+1} = z_{n+1}-q_{n+1}.
			\end{array}
			\right.
		\end{array}
	\end{equation}
\end{algo}
Algorithm~\eqref{algo:algo1} combines 
Condat-V\~u and FBHF as follows. 

Note that, in the particular case when $C=0$, \eqref{eq:algo1} 
reduces to the Condat-V\~u algorithm:
\begin{equation}\label{eq:algoCV}
	(\forall n\in\N)\quad 
	\begin{array}{l}
		\left\lfloor
		\begin{array}{l}
			x_{n+1} = J_{\tau A} (x_n-\tau (L^* u_n+D 
			x_n) 
			) \\
			u_{n+1} =  J_{\sigma
				B^{-1}}\left(u_{n}+\sigma L( 
			2x_{n+1}-x_{n})\right).
		\end{array}
		\right.
	\end{array}
\end{equation}
On the other hand, when $B=0$ and $L=0$, 
the recurrence in \eqref{eq:algo1} reduces to the FBHF algorithm:
\begin{equation}\label{eq:FBHF}
	(\forall n\in\N)\quad 
	\begin{array}{l}
		\left\lfloor
		\begin{array}{l}
			z_{n+1} = J_{\tau A} (x_n-\tau (C{ x_{n}}+D 
			x_n))\\
			x_{n+1} = z_{n+1}-\tau(Cz_{n+1}-C x_n).
		\end{array}
		\right.
	\end{array}
\end{equation}
The weak convergence of the recurrence in \eqref{eq:algoCV} to a zero of $A+L^*BL+C$ is guaranteed for 
step-sizes $(\tau,\sigma)$ satisfying $
\sigma\tau\|L\|^2<1-\frac{\tau}{2\beta}$ \cite[Theorem~3.1]{Condat13}.
Additionally, the weak convergence of \eqref{eq:FBHF} to a zero of $A+C+D$ is guaranteed for $\tau \in \,
]0,{4\beta}/{(1+\sqrt{1+16\beta^2\zeta^2)}}[$.
Our main result guarantees the weak convergence of Algorithm~\eqref{algo:algo1} to a solution of Problem~\ref{algo:algo1} and additionally recover the conditions on step-sizes on Condat-V\~u and FBHF. 

{ 
	Note that, in the context of convex optimization in 
	Problem~\ref{pro:opti}, Algorithm~\ref{algo:algo1}
	can be written as follows:
	\begin{equation}\label{eq:algo1opti}
		(\forall n\in\N)\quad 
		\begin{array}{l}
			\left\lfloor
			\begin{array}{l}
				p_{n+1} = \nabla h (x_n)\\
				z_{n+1} = \prox_{\tau g} (x_n-\tau (L^* 
				u_n+p_{n+1}+\nabla d 
				(x_n)) 
				) \\
				q_{n+1} = \tau(\nabla h (z_{n+1})-p_{n+1})\\
				u_{n+1} =  \prox_{\sigma
					g^*}\left(u_{n}+\sigma L( 
				2z_{n+1}-x_{n}-q_{n+1})\right)\\
				x_{n+1} = z_{n+1}-q_{n+1}.
			\end{array}
			\right.
		\end{array}
	\end{equation}
	Finally, note that the primal-dual extension of FBHF proposed in \cite[Section~5]{BricenoDavis2018} for solving Problem~\ref{pro:main}, requires one extra {  evaluation} of $L$ and $L^*$ per iteration compared to Algorithm~\ref{algo:algo1}.}	
\section{Preliminaries and Notation}\label{sec:Pre}
In this article, $\H$ and $\G$ are real Hilbert spaces with scalar 
product $\scal{\cdot}{\cdot}$ and norm $\|\cdot \|$. {  Let $X$ and $Y$ be subsets of $\H$, $X\times Y$ denotes the Cartesian product of $X$ and $Y$. We use the notation $X^2=X\times X$}. We denoted 
by $\to$ the strong convergence and by $\weakly$ the weak 
convergence. The identity operator is 
denoted by $\id$. Given a linear operator $L:\H\to\G$ we denote 
its adjoint by $L^* \colon \G\to\H$ and its norm, indistinctly, 
by $\|L\|:=\sup_{x \in \H}(\|Lx\|/\|x\|)$. Note that, for every $x \in \H$, $\|Lx\|\leq \|L\|\|x\|$. Consider the
following classic inequalities, 
\begin{eqnarray}\label{eq:CS}
	&\hspace*{-1cm}(\textnormal{Cauchy–Schwarz inequality})  &(\forall 
	(x,u) \in 
	\H^2) \quad 
	|\scal{x}{u}| \leq \|x\|\|u\|,\\
	\label{eq:YI}&(\textnormal{Young's inequality})   &(\forall (a,b)\in 
	\R^2)(
	\forall \epsilon 
	>0) \quad 2ab \leq 
	\epsilon a^2+\frac{b^2}{\epsilon}.
\end{eqnarray} 	
Let $T\colon \H \rightarrow 
\H$ and $\beta \in \left]0,+\infty\right[$. The operator $T$ is 
$\beta$-cocoercive if 
\begin{equation} \label{def:coco}
	(\forall x \in \H) (\forall y \in \H)\quad \langle x-y \mid 
	Tx-Ty 
	\rangle 
	\geq \beta \|Tx - Ty 
	\|^2.
\end{equation}
Additionally, the operator $T$ is $\beta$-Lipschitz if 
\begin{equation} \label{def:lips}
	(\forall x \in \H) (\forall y \in \H)\quad \|Tx-Ty\| \leq  
	\beta\|x - y \|.
\end{equation}	
Let $A\colon\H \rightarrow 2^{\H}$ be a set-valued operator.
The domain, range, zeros, and graph of $A$ 
are defined, respectively, by:
\begin{align*}
	&\dom\, A = \menge{x \in \H}{Ax \neq  \varnothing}, \quad 
	\ran\, A = \menge{u \in \H}{(\exists x \in \H)\,\, u \in Ax}\\
	&\zer A = 	\menge{x \in \H}{0 \in Ax}, \quad
	\gra A = \menge{(x,u) \in \H \times \H}{u \in Ax}.
\end{align*}
The inverse of the $A$ is defined by
$A^{-1}\colon\H \rightarrow 2^{\H} \colon u \mapsto \menge{x \in \H}{u \in Ax}$.
Let $\rho \in \R$, the operator $A$ is called $\rho$-monotone if 
\begin{equation}\label{eq:defrhomon}
	\left(\forall \big((x,u),(y,v)\big) \in (\gra A)^2\right) 
	\quad 
	\scal{x-y}{u-v} 
	\geq \rho \|x-y\|^2.
\end{equation}
Moreover, $A$ is maximally $\rho$-monotone if it is 
$\rho$-monotone and its graph is 
maximal in the sense of 
inclusions among the graphs of $\rho$-monotone operators. In the 
case when $\rho=0$, $A$ is (maximally) monotone, and when $\rho 
>0$ $A$ is strongly (maximally) monotone.
The resolvent of a maximally $\rho$-monotone operator $A$ is 
defined by {  $J_A\colon\H \rightarrow 2^{\H} \colon x \mapsto (\id+A)^{-1}x$}. If $\rho>-1$, $J_A$ is single valued
\cite{BauschkeMoursiXianfu2021}. Note that, if $A$ is 
$\rho$-monotone, then, for every $\gamma \in \RPP$, $\gamma A$ is 
$\gamma\rho$-monotone. If $A$ is maximally monotone, then, its 
inverse $A^{-1}$ is also a maximally monotone operator.
In the context of Problem~\ref{pro:main}, since  $A+C$ is 
maximally monotone and $D$ is monotone and continuous with full 
domain, it follows from \cite[Corollary~20.28 \& 
Corollary~25.5(i)]{bauschkebook2017} 
that $A+C+D$ is maximally monotone. Hence, it follows  from 
\cite[Proposition 20.23]{bauschkebook2017} that the operator
\begin{equation}\label{eq:defbA}
	\boldsymbol{A}: \H \oplus\G \to 
	2^{\H \oplus\G} \colon (x,u) \mapsto (A+C)x \times 
	B^{-1}u,
\end{equation}
is maximally monotone. Moreover, it follows from 
the maximal monotonicity of $A+C+D$ and 
\cite[Proposition~2.7(iii)]{briceno2011SIAM} that the operator 
\begin{equation}\label{eq:defbM}
	\boldsymbol{M}: \H 
	\oplus\G \to 
	2^{\H \oplus\G} \colon (x,u) \mapsto ((A+D+C)x+L^*u) 
	\times 
	(B^{-1}u-Lx)
\end{equation}
is maximally monotone. Note that $\zer \bm{M} = \bm{Z}$.

We denote by $\Gamma_0(\H)$ the class of proper lower 
semicontinuous convex functions $f\colon\H\to\RX$. Let 
$f\in\Gamma_0(\H)$.
The Fenchel conjugate of $f$ is 
defined by $f^*\colon u\mapsto \sup_{x\in\H}(\scal{x}{u}-f(x))$. We have
$f^*\in \Gamma_0(\H)$. The subdifferential of $f$ is defined by
$$\partial f\colon x\mapsto \menge{u\in\H}{(\forall y\in\H)\:\: 
	f(x)+\scal{y-x}{u}\le f(y)},$$ we have that $\partial f$ is a maximally monotone operator and
$(\partial f)^{-1}=\partial f^*$ 
and that $\zer\,\partial f$ is the set of 
minimizers of $f$, which is denoted by $\arg\min_{x\in \H}f$. 	
The proximity operator of $f$ is $
\label{e:prox}
\prox_{f}\colon 
x\mapsto\textnormal{arg min}_{y\in\H}\left(f(y)+\frac{1}{2}\|x-y\|^2\right)$.
We have $\prox_f=J_{\partial f}$.

For a further background on monotone operators and convex analysis, 
the reader is referred to \cite{bauschkebook2017}.

\section{Convergence Theory}\label{sec:conv}
This section is dedicated to studying the convergence of 
Algorithm~\ref{algo:algo1}. First, we present 
the following proposition, providing key inequalities to 
guarantee the weak convergence of our method.
\begin{prop}\label{prop:des}
	In the context of Problem~\ref{pro:main}, let $(x_0,u_0)\in 
	\H\times \G$, let $(\sigma,\tau)\in \RPP^2$, suppose that 
	$\tau\rho>-1$, and let 
	$(z_n)_{n \in 
		\N}$, $(x_n)_{n \in \N}$, and $(u_n)_{n \in \N}$ be { the} 
	sequences 
	generated by Algorithm~\ref{algo:algo1}.
	For each $n \in \N$ and each $(x,u) \in \bm{Z}$, define
	\begin{equation}\label{def:gamman}
		\Gamma_{n}(x,u):=\|x_{n}-x\|^2 
		+\frac{\tau}{\sigma}\|u_{n}-u\|^2-2\tau\scal{x_{n}-x}{L^*(u_{n}-u)}.
	\end{equation}
	The following assertions hold:
	
	\begin{enumerate}
		\item \label{prop:des1} 
		{  For every $(x,u) \in \bm{Z}$, every $(\varepsilon,\kappa) \in \RPP^2$}, and every $n \in \N$
		\begin{align}\label{eq:propdes1} 
			\Gamma_{n+1}(x,u) &\leq 
			\Gamma_{n}(x,u)-(1-\varepsilon-\tau^2\zeta^2-\kappa\tau\sigma\|L\|^2)\|z_{n+1}-x_n\|^2
			\nonumber\\
			&\hspace{0.8cm}-\frac{\tau}{\sigma}\left(1-\frac{1}{\kappa}\right)\|u_{n+1}-u_n\|^2
			-\frac{\tau}{\varepsilon}\left(2\beta\varepsilon-\tau\right)\|Dx_n-Dx\|^2.
		\end{align}
		\item \label{prop:des2}	{  For every $(x,u) \in \bm{Z}$ and}  every $n \in \N$
		\begin{equation}\label{eq:propdes2} 
			\Gamma_n(x,u) \geq 
			(1-\sigma\tau\|L\|^2)\max\left\lbrace\|x_{n}-x\|^2,\frac{\tau}{\sigma}\|u_{n}-u\|^2\right\rbrace.
		\end{equation}
	\end{enumerate}
\end{prop}
\begin{proof}
\begin{enumerate}
	\item 
	Let $(x,u) \in \bm{Z}$ and fix $n \in \N$. 
	It follows from \eqref{eq:algo1} that
	\begin{align*}
		\begin{cases}
			\dfrac{x_n-z_{n+1}}{\tau}-L^*u_n- C x_n -  
			Dx_n + C z_{n+1} \in (A+C)
			z_{n+1}\\
			\dfrac{u_n-u_{n+1}}{\sigma}+L(2z_{n+1}-x_n-\tau(Cz_{n+1}-Cx_n))
			\in B^{-1}u_{n+1}.
		\end{cases}
	\end{align*}	
	Then, since $-L^*u-Dx \in (A+C)x$ and $Lx \in B^{-1}u$,  
	by 
	the monotonicity of the operator $\boldsymbol{A}$
	defined in \eqref{eq:defbA}, we have that
	\begin{align}\label{eq:mon}
		0	\leq &\scal{\frac{x_n-z_{n+1}}{\tau}-L^*u_n- C 
			x_n -  
			Dx_n +Cz_{n+1}
			+L^*u+Dx}{z_{n+1}-x}\nonumber\\
		&\quad +
		\scal{\frac{u_n-u_{n+1}}{\sigma}+L(2z_{n+1}-x_n-\tau(Cz_{n+1}-Cx_n)-x)}
		{u_{n+1}-u}.
	\end{align}	
	Note that
	\begin{align*}
		\frac{1}{\tau}\scal{x_n-z_{n+1}}{z_{n+1}-x}
		=\frac{1}{2\tau}(\|x_n-x\|^2-\|x_n-z_{n+1}\|^2-\|z_{n+1}-x\|^2)
	\end{align*}
	and
	\begin{align*}
		\frac{1}{\sigma}\scal{u_n-u_{n+1}}
		{u_{n+1}-u}
		=\frac{1}{2\sigma}(\|u_n-u\|^2-\|u_n-u_{n+1}\|^2-\|u_{n+1}-u\|^2).
	\end{align*}
	%		Now, by the monotonicity of $C$
	%		\begin{align*}
		%			-\scal{Cx_n-Cx}{z_{n+1}-x} &= 
		%			
		%-\scal{Cz_{n+1}-Cx}{z_{n+1}-x}-\scal{Cx_n-Cz_{n+1}}{z_{n+1}-x}\\
		%			&\leq -\scal{Cx_n-Cz_{n+1}}{z_{n+1}-x}.
		%		\end{align*}
	Moreover, by the cocoercivity of $D$,{  \eqref{eq:CS}, and \eqref{eq:YI}}, for every 
	$\varepsilon 
	>0$, we have
	\begin{align*}
		\scal{Dx_n-Dx}{x-z_{n+1}} &=  
		\scal{Dx_n-Dx}{x-x_{n}}+\scal{Dx_n-Dx}{x_{n}-z_{n+1}}\\
		&\leq 
		-\beta\|Dx_n-Dx\|^2+\scal{Dx_n-Dx}{x_{n}-z_{n+1}}\\
		%			&\leq 
		%			-\beta\|Dx_n-Dx\|^2+\frac{\tau}{2\varepsilon}\|Dx_n-Dx\|^2
		%			+\frac{\varepsilon}{2\tau}
		%			\|x_{n}-z_{n+1}\|^2\\
		&\leq \left(	
		\frac{\tau}{2\varepsilon}-\beta\right)\|Dx_n-Dx\|^2+\frac{\varepsilon}{2\tau}
		\|x_{n}-z_{n+1}\|^2.
	\end{align*}
	Hence, it follows from \eqref{eq:mon} that
	\begin{align}\label{eq:mon2}
		\|z_{n+1}&-x\|^2 
		+\frac{\tau}{\sigma}\|u_{n+1}-u\|^2\nonumber\\
		&\leq 
		\|x_n-x\|^2-(1-\varepsilon)\|z_{n+1}-x_n\|^2 + 
		\frac{\tau}{\sigma}(\|u_n-u\|^2-\|u_n-u_{n+1}\|^2)\nonumber\\
		&\quad 
		\quad-\frac{\tau}{\varepsilon}\left(2\beta\varepsilon-\tau\right)\|Dx_n-Dx\|^2
		+2\tau\scal{Cz_{n+1}-Cx_n}{z_{n+1}-x}\nonumber\\
		&\quad 
		\quad+2\tau\scal{2z_{n+1}-x_n-\tau(Cz_{n+1}-Cx_n)-x}
		{L^*(u_{n+1}-u)}\nonumber\\
		& \quad \quad -2\tau\scal{L^*(u_n-u)}{z_{n+1}-x}.
	\end{align}
	Additionally, from \eqref{eq:algo1} { and the Lipschitz property of $C$}, we deduce
	\begin{align*}
		\|x_{n+1}-x\|^2 &= 
		\|z_{n+1}-\tau(Cz_{n+1}-Cx_n)-x\|^2\\
		&=\|z_{n+1}-x\|^2-2\tau\scal{z_{n+1}-x}{Cz_{n+1}-Cx_n}+\tau^2\|Cz_{n+1}-Cx_n\|^2\\
		&\leq 
		\|z_{n+1}-x\|^2-2\tau\scal{z_{n+1}-x}{Cz_{n+1}-Cx_n}
		+\tau^2\zeta^2\|z_{n+1}-x_n\|^2.
	\end{align*}
	Therefore,  from \eqref{eq:mon2} we conclude that
	\begin{align}\label{eq:mon3}
		\|x_{n+1}-x\|^2 
		+&\frac{\tau}{\sigma}\|u_{n+1}-u\|^2\nonumber\\
		&\leq 
		\|x_n-x\|^2-(1-\varepsilon-\tau^2\zeta^2)\|z_{n+1}-x_n\|^2\nonumber\\
		&\hspace*{1cm}	
		+\frac{\tau}{\sigma}(\|u_n-u\|^2-\|u_n-u_{n+1}\|^2)\nonumber\\
		&\hspace*{1cm}-\frac{\tau}{\varepsilon}\left(2\beta\varepsilon-\tau\right)\|Dx_n-Dx\|^2
		-2\tau\scal{L^*(u_n-u)}{z_{n+1}-x}\nonumber\\
		&\hspace*{1cm}+2\tau\scal{2z_{n+1}-x_n-\tau(Cz_{n+1}-Cx_n)-x}
		{L^*(u_{n+1}-u)}.
	\end{align}
	By summing the last two terms of \eqref{eq:mon3}, we 
	obtain			
	\begin{align}\label{eq:mon4}	
		&\scal{2z_{n+1}-x_n-\tau(Cz_{n+1}-Cx_n)-x}{L^*(u_{n+1}-u)}
		-\scal{L^*(u_n-u)}{z_{n+1}-x}\nonumber\\
		&=\scal{z_{n+1}+x_{n+1}-x_n-x}{L^*(u_{n+1}-u)}
		-\scal{L^*(u_n-u)}{z_{n+1}-x}\nonumber\\
		&=\scal{x_{n+1}-x}{L^*(u_{n+1}-u)}+\scal{z_{n+1}-x_n}{L^*(u_{n+1}-u)}
		-\scal{L^*(u_n-u)}{z_{n+1}-x}\nonumber\\
		&=\scal{x_{n+1}-x}{L^*(u_{n+1}-u)}-\scal{x_n-x}{L^*(u_n-u)}
		+\scal{z_{n+1}-x_n}{L^*(u_{n+1}-u_n)}.
	\end{align}
	Moreover, { \eqref{eq:CS} and \eqref{eq:YI} yield}
	\begin{align}\label{eq:mon5}	
		2\tau\scal{z_{n+1}-x_n}{L^*(u_{n+1}-u_n)}&\leq 
		2\tau\|L\|\|z_{n+1}-x_n\|\|u_{n+1}-u_n\|\nonumber\\
		&\leq 
		\kappa\tau\sigma\|L\|^2\|z_{n+1}-x_n\|^2
		+\frac{\tau}{\sigma\kappa}\|u_{n+1}-u_n\|^2.
	\end{align}	
	Therefore, it follows from 
	\eqref{eq:mon3}-\eqref{eq:mon5} that 
	%\begin{align*}
	%	\|x_{n+1}-x\|^2 
	%	
	%+\frac{\tau}{\sigma}\|u_{n+1}-u\|^2-&2\tau\scal{x_{n+1}-x}{L^*(u_{n+1}-u)}\\
	%	&\leq
	%	
	%\|x_n-x\|^2+\frac{\tau}{\sigma}\|u_n-u\|^2-2\tau\scal{x_n-x}{L^*(u_n-u)}\\
	%	&-(1-\varepsilon-\tau^2\zeta^2)\|z_{n+1}-x_n\|^2
	%	-\frac{\tau}{\sigma}\|u_{n+1}-u_n\|^2\nonumber\\
	%	
	%&-\frac{\tau}{\varepsilon}\left(2\beta\varepsilon-\tau\right)\|Dx_n-Dx\|^2
	%	+2\tau\scal{z_{n+1}-x_n}{L^*(u_{n+1}-u_n)}.
	%\end{align*}
	\begin{align*}
		\|x_{n+1}-x\|^2 
		+\frac{\tau}{\sigma}&\|u_{n+1}-u\|^2-2\tau\scal{x_{n+1}-x}{L^*(u_{n+1}-u)}\\
		&\leq
		\|x_n-x\|^2+\frac{\tau}{\sigma}\|u_n-u\|^2-2\tau\scal{x_n-x}{L^*(u_n-u)}\\
		&\hspace{0.5cm}-(1-\varepsilon-	
		\tau^2\zeta^2-\kappa\tau\sigma\|L\|^2)\|z_{n+1}-x_n\|^2
		\\
		&\hspace{0.5cm}-\frac{\tau}{\sigma}\left(1-\frac{1}{\kappa}\right)\|u_{n+1}-u_n\|^2
		-\frac{\tau}{\varepsilon}\left(2\beta\varepsilon-\tau\right)\|Dx_n-Dx\|^2.
	\end{align*}
	The result follows.
	\item It follow from { \eqref{eq:CS} and \eqref{eq:YI}} that
	\begin{align*}
		(\forall n \in \N) \quad  \Gamma_{n}{ (x,u)}&=\|x_{n}-x\|^2 
		+\frac{\tau}{\sigma}\|u_{n}-u\|^2-2\tau\scal{x_{n}-x}{L^*(u_{n}-u)}\\
		&\geq \|x_{n}-x\|^2 
		+\frac{\tau}{\sigma}\|u_{n}-u\|^2 - \sigma\tau 
		\|L\|^2\|x_n-x\|^2-\frac{\tau}{\sigma}\|u_n-u\|^2\\
		&\geq (1-\sigma\tau\|L\|^2)\|x_n-x\|^2.
	\end{align*}
	Analogously, we conclude that
	\begin{align*}
		(\forall n \in \N) \quad  \Gamma_{n}{ (x,u)}\geq 
		\frac{\tau}{\sigma}(1-\sigma\tau\|L\|^2)\|u_n-u\|^2.
	\end{align*}
	The results follows.
\end{enumerate}	
\end{proof}
The following theorem is our main result and provides the weak 
convergence of Algorithm~\ref{algo:algo1} to a solution point in 
$\bm{Z}$.
\begin{teo}\label{teo:conve1}
	In the context of Problem~\ref{pro:main}, let $(\sigma,\tau)\in \RPP^2$,  suppose that 
	$\tau\rho>-1$, let $\varepsilon 
	\in 
	\,]0,1[$, let $(x_0,u_0)\in 
	\H\times \G$, and let 
	$(x_n)_{n \in 
		\N}$ and $(u_n)_{n \in \N}$ be the sequences 
	generated by Algorithm~\ref{algo:algo1}. Suppose that	
	\begin{equation}\label{eq:steps1}
		\tau \leq 2\beta\varepsilon
	\end{equation}	
	and that
	\begin{equation}\label{eq:steps2}
		\tau \sigma \|L\|^2 +\tau^2\zeta^2< 1-\varepsilon.
	\end{equation}
	{ Then}, the following statements hold
	\begin{enumerate}
		\item $\sum_{k=0}^\infty\|z_{k+1}-x_k\|^2<+\infty$, 
		$\sum_{k=0}^\infty\|x_{k+1}-x_k\|^2<+\infty$, and 
		$\sum_{k=0}^\infty\|u_{k+1}-u_k\|^2<+\infty$.
		\item There exist $(x,u) \in \bm{Z}$ such that 
		$(x_n,u_n)\weak (x,u)$.	
	\end{enumerate}
	
\end{teo}
\begin{proof}
Let $(x,u) \in \bm{Z}$. Since $2\beta\varepsilon-\tau \geq 
0$, it 
follows from 
Proposition~\ref{prop:des}.\ref{prop:des1} that for every 
$\kappa \in 
\RPP$ and every $n \in \N$
\begin{align}\label{eq:proofc0}
	\Gamma_{n+1}(x,u)\leq 
	\Gamma_{n}(x,u)
	-(1-\varepsilon-&\tau^2\zeta^2-\kappa\tau\sigma\|L\|^2)\|z_{n+1}-x_n\|^2\nonumber\\
	&-\frac{\tau}{\sigma}\left(1-\frac{1}{\kappa}\right)\|u_{n+1}-u_n\|^2.
\end{align}
In particular, for $\kappa=1$
\begin{equation}\label{eq:proofc1}
	(n \in \N)\quad \Gamma_{n+1}(x,u) \leq 
	\Gamma_{n}(x,u)
	-(1-\varepsilon-\tau^2\zeta^2-\tau\sigma\|L\|^2)\|z_{n+1}-x_n\|^2.
\end{equation}
Since, for every $n \in \N$, $\Gamma_n(x,u) \geq 0$ 
(Proposition~\ref{prop:des}.\ref{prop:des2}), it follows from 
\eqref{eq:steps2} and \cite[Lemma~3.1]{Comb2001} that for all 
$ 
(x,u)\in \bm{Z}$, $(\Gamma_n(x,u))_{n \in \N}$ is convergent 
and  
\begin{equation}\label{eq:zxto0}
	\sum_{k=0}^\infty\|z_{k+1}-x_k\|^2<+\infty.
\end{equation}
Note that, {  the Lipschitz property of $C$ yields}
\begin{align*}
	(n \in \N)\quad 	\|x_{n+1}-x_n\|&= \|(\id-\tau C )z_{n+1}-(\id-\tau C 
	)x_{n}\|\\
	&			\leq (1+\tau\zeta)\|z_{n+1}-x_{n}\|,
\end{align*}
then, $\sum_{k=0}^\infty\|x_{k+1}-x_k\|^2<+\infty$. Moreover, 
taking 
$\kappa=2$ in \eqref{eq:proofc0}, we have, for all $n \in \N$, that
\begin{align*}
	\Gamma_{n+1}(x,u) &\leq 
	\Gamma_{n}(x,u)
	-\frac{\tau}{2\sigma}\|u_{n+1}-u_n\|^2
	+\big|2\tau\sigma\|L\|^2+\varepsilon+\tau^2\zeta^2-1\big|\|z_{n+1}-x_n\|^2.
\end{align*}
Then, by 
\cite[Lemma~3.1]{Comb2001} we conclude
\begin{equation}\label{eq:uuto0}
	\sum_{k=0}^\infty\|u_{k+1}-u_k\|^2<+\infty.
\end{equation}
Now, by telescoping \eqref{eq:proofc1},
\begin{equation}
	(\forall n \in \N) \quad\Gamma_{n+1}(x,u) \leq 
	\Gamma_{0} (x,u)
	-(1-\varepsilon-\tau^2\zeta^2-\tau\sigma\|L\|^2)
	\sum_{k=0}^n\|z_{k+1}-x_k\|^2.
\end{equation}
Hence, in view of \eqref{eq:steps2} and 
Proposition~\ref{prop:des}.\ref{prop:des2} we deduce
\begin{equation}
	(\forall n \in \N) \quad	(\varepsilon+\tau^2\zeta^2)\max\left\lbrace\|x_{n}-x\|^2,\frac{\tau}{\sigma}\|u_{n}-u\|^2\right\rbrace
	\leq \Gamma_0(x,u).
\end{equation}
Therefore, $(x_n)_{n \in \N}$ and $(u_n)_{n \in \N}$ are 
bounded. Furthermore, in view of \eqref{eq:zxto0}, \eqref{eq:uuto0}, the continuity 
of 
$L^*$, 
and the Lipschitz property of $C$ and $D$, we conclude 
\begin{equation}\label{eq:Czxto0}
	\|L^*(u_{n+1}-u_{n})\| \to 
	0, 
	\ \|Cx_{n}-Cz_{n+1}\| 
	\to 0, \text{ and } \|Dx_{n}-Dz_{n+1}\| \to 0.
\end{equation}
Hence, we have that
\begin{align}\label{eq:rto0}
	&r_{n}:=\dfrac{x_{n}-z_{{n}+1}}{\tau}-L^*(u_{n}-u_{{n}+1})
	- (C x_{n}- C 
	z_{{n}+1})-  
	(Dx_{n}-Dz_{{n}+1})\to 0,\\
	&w_{n} := 
	\dfrac{u_{n}-u_{{n}+1}}{\sigma}+L(z_{{n}+1}-x_{n}-\tau(Cz_{{n}+1}-Cx_{n}))
	\to 0.\label{eq:wto0}
\end{align}
Now, let $(\hat{x},\hat{u})$ be a weak cluster point of 
$(x_n,u_n)_{n 
	\in \N}$. Thus, there exist a subsequence 
$(x_{n_k},u_{n_k})_{k 
	\in 
	\N}$, such that $(x_{n_k},u_{n_k})\to(\hat{x},\hat{u})$.
Hence, by \eqref{eq:zxto0}
\begin{equation}\label{eq:ztoxweak}
	z_{n_k+1} \weak \hat{x}.
\end{equation}
It follows from \eqref{eq:algo1} that
%\begin{equation}\label{eq:maxmon}
%	\begin{cases}
	%		
	%\dfrac{x_{n_k}-z_{{n_k}+1}}{\tau}-L^*(u_{n_k}-u_{{n_k}+1})
	%		- (C x_{n_k}- C 
	%		z_{{n_k}+1})-  
	%		(Dx_{n_k}-Dz_{{n_k}+1})\in (A 
	%		+C+D)z_{{n_k}+1}+L^*u_{{n_k}+1}\\
	%		
	%\dfrac{u_{n_k}-u_{{n_k}+1}}{\sigma}+L(z_{{n_k}+1}-x_{n_k}-\tau(Cz_{{n_k}+1}-Cx_{n_k}))
	%		\in B^{-1}u_{{n_k}+1}-Lz_{{n_k}+1}.
	%	\end{cases}
%\end{equation}	
\begin{equation}\label{eq:maxmon}
	(r_{n_k},w_{n_k}) \in 
	\boldsymbol{M}(z_{{n_k}+1},u_{{n_k}+1}),
\end{equation}	
where 
$\boldsymbol{M}$ is the maximally monotone operator defined 
in 
\eqref{eq:defbM}.
%Since $C+D$ is monotone and continuous with full domain, it 
%follows 
%from \cite[Corollary~20.28 \& 
%Corollary~25.5(i)]{bauschkebook2017} 
%that $A+C+D$ is maximally monotone. Moreover, it follows 
%from 
%\cite[Proposition~2.7(iii)]{briceno2011SIAM} that 
%$\boldsymbol{M}: 
%\H 
%\oplus\G \to 
%2^{\H \oplus\G} \colon (x,u) \mapsto ((A+D+C)x+L^*u) \times 
%(B^{-1}u-Lx)$ 
%is maximally monotone. 
Hence, in view of \eqref{eq:Czxto0}-\eqref{eq:maxmon} and the 
weak-strong closure of $\gra(\boldsymbol{M})$ 
\cite[Proposition~20.38(ii)]{bauschkebook2017}, we 
conclude that 
$(\hat{x},\hat{u})\in \zer(\bm{M}) = \bm{Z}$. In view of \cite[Lemma 
2.46]{bauschkebook2017}, we will conclude the proof showing 
that 
$(x_n,u_n)_{n \in \N}$ possesses at most one weak sequential
cluster point. Let $(x,u) \in \bm{Z}$ and $(y,v)\in \bm{Z}$ 
be 
cluster points of 
$(x_n,u_n)_{n \in \N}$ such that $(x_{n_k},u_{n_k})\weak 
(x,u)$ 
and $(x_{n_m},u_{n_m})\weak (y,v)$. Since, for all $ 
(\hat{x},\hat{u})\in \bm{Z}$, $(\Gamma_n(\hat{x},\hat{u}))_{n 
	\in 
	\N}$ is convergent, we assume
that $\Gamma_n(x,u)\to\ell_1$ and $\Gamma_n(y,v)\to\ell_2$. 
Note 
that, for every $n \in \N$,
\begin{align*}
	\sigma(&\Gamma_n((x,u))-\Gamma_n(y,v))\\
	&= 
	\sigma\|x_{n}-x\|^2 
	+\tau\|u_{n}-u\|^2-2\tau\sigma\scal{x_{n}-x}{L^*(u_{n}-u)}\\
	&\hspace*{1cm}-\sigma\|x_{n}-y\|^2 
	-\tau\|u_{n}-v\|^2+2\tau\sigma\scal{x_{n}-y}{L^*(u_{n}-v)}\\
	%	&= \sigma\|x_n\|^2-2\sigma 
	%\scal{x_n}{x}+\sigma\|x\|^2 +
	%	\tau\|u_n\|^2-2\tau 
	%	
	%\scal{u_n}{u}+\tau\|u\|^2-2\tau\sigma\scal{x_{n}-x}{L^*(u_{n}-u)}\\
	%	&- \sigma\|x_n\|^2+2\sigma 
	%	\scal{x_n}{y}-\sigma\|y\|^2-\tau\|u_n\|^2+2\tau 
	%	
	%\scal{u_n}{v}-\tau\|v\|^2+2\tau\sigma\scal{x_{n}-y}{L^*(u_{n}-v)}\\
	&= 
	2\sigma\scal{x_n}{y-x}+\sigma(\|x\|^2-\|y\|^2)
	+2\tau\scal{u_n}{v-u}+\tau(\|u\|^2-\|v\|^2)\\
	&\hspace*{2cm}
	+2\tau\sigma(\scal{x_n-x}{L^*u}+\scal{y-x_n}{L^*v}+\scal{x-y}{L^*u_n}).
\end{align*}
In particular, taking $n=n_k\to { +}\infty$, we obtain
\begin{equation}\label{eq:ell1ell21}
	\ell_1-\ell_2=-\sigma\|x-y\|^2-\tau\|u-y\|^2+2\tau\sigma\scal{y-x}{L^*(v-u)}.
\end{equation}
Analogously, taking $n=n_m\to{ +}\infty,$
\begin{equation}\label{eq:ell1ell22}
	\ell_1-\ell_2= \sigma\|x-y\|^2+\tau\|u-v\|^2
	-2\tau\sigma\scal{y-x}{L^*(v-u)}.
\end{equation}
Therefore, {by  \eqref{eq:CS} and \eqref{eq:YI}}, we conclude that 
\begin{align*}
	0&=\sigma\|x-y\|^2+\tau\|u-v\|^2-2\tau\sigma\scal{y-x}{L^*(v-u)}\\
	&\geq\sigma\|x-y\|^2+\tau\|u-v\|^2-\tau\sigma^2\|L\|^2\|x-y\|^2-\tau\|u-v\|^2\\
	&= \sigma(1-\sigma\tau\|L\|^2)\|x-y\|^2\geq 0.
\end{align*}
%\begin{align*}
%	
%0&=\sigma\|x-y\|^2+\tau\|u-v\|^2-2\tau\sigma\scal{y-x}{L^*(v-u)}\\
%	
%&\geq\sigma\|x-y\|^2+\tau\|u-v\|^2-\sigma\|x-y\|^2-\sigma\tau^2\|L\|^2\|u-v\|^2\\
%	&= \tau(1-\sigma\tau\|L\|^2)\|u-v\|^2\geq 0.
%\end{align*}
Hence, $x=y$. By a symmetric argument we deduce 
that 
\begin{equation}
	\tau(1-\sigma\tau\|L\|^2)\|u-v\|^2=0,
\end{equation}
thus $v=u$, and the results follows.
\end{proof}

\begin{rem}\label{rem:teo1}
	\begin{enumerate}
		\item In the case when $C=0$ Algorithm~\ref{algo:algo1} 
		reduces to the Condat-V\~u algorithm. Furthermore, by 
		taking $\zeta=0$ and choosing $\varepsilon = 
		\tau/(2\beta)$, \eqref{eq:steps2} reduces to
		\begin{equation}
			\sigma\tau\|L\|^2<1-\frac{\tau}{2\beta},
		\end{equation}
		which correspond to the step-size requirement proposed in \cite{Condat13,Vu13} for ensuring the weak convergence of the Condat-V\~u algorithm.
		\item In the case when $B=0$ and $L=0$, the dual sequence 
		$(u_n)_{n\in \N}$ 
		{  is zero}. Then, Algorithm~\ref{algo:algo1} 
		reduces to the FBHF algorithm. Moreover,
		\eqref{eq:steps2} reduces to
		\begin{equation}
			\tau^2\zeta^2<1-\varepsilon.
		\end{equation}
		By choosing $\varepsilon$ such that 
		$2\beta\zeta\varepsilon=\sqrt{1-\varepsilon}$, we recover 
		the condition on the step-size,
		\begin{equation}
			\tau \in 
			\left]0,\frac{4\beta}{1+\sqrt{1+16\beta^2\zeta^2}}\right[,
		\end{equation}	
		for guaranteeing the weak convergence of FBHF 
		proposed in 
		\cite[Theorem~2.3]{BricenoDavis2018}. 
		Additionally, FBHF reduces to Tseng’s splitting
		and FB when $D=0$ and $C=0$, respectively 
		\cite[Remark~4]{BricenoDavis2018}.
		%\item Si $D=0$ y $L=\id$ recupero?
		%\item 
	\end{enumerate}
\end{rem}	
To conclude this section, we present the following result as a direct consequence of Theorem~\ref{teo:conve1} that provides an algorithm to solve
Problem~\ref{pro:main} when $D=0$. {  This algorithm combines PD and FBF, which is also a new result that integrates these two well-known methods}.
\begin{cor}
	In the context of Problem~\ref{pro:main}, let $(x_0,u_0)\in 
	\H\times \G$, let $(\sigma,\tau)\in \RPP^2$, assume that 
	$\tau\rho >-1$, and consider the following recurrence
	\begin{equation}\label{eq:algocoro}
		(\forall n\in\N)\quad 
		\begin{array}{l}
			\left\lfloor
			\begin{array}{l}
				p_{n+1} = C x_n\\
				z_{n+1} = J_{\tau A} (x_n-\tau (L^* u_n+p_{n+1}) 
				) \\
				q_{n+1} = \tau(Cz_{n+1}-p_{n+1})\\
				u_{n+1} =  J_{\sigma
					B^{-1}}\left(u_{n}+\sigma L( 
				2z_{n+1}-x_{n}-q_{n+1})\right)\\
				x_{n+1} = z_{n+1}-q_{n+1}.
			\end{array}
			\right.
		\end{array}
	\end{equation}
	Suppose that \begin{equation}\label{eq:stepscoro}
		\tau \sigma \|L\|^2 +\tau^2\zeta^2<1.
	\end{equation}
	Then, $(x_n,u_n)\weak 
	(x,u)\in \H\times \G$ such that $0 \in 
	(A+C)x-L^*u$ and $0\in B^{-1}u+Lx$.
\end{cor}	
\section{Applications}\label{sec:ext}
In this section, we present applications of our algorithm. First, we present an algorithm for solving multivariate monotone inclusions. Next, we apply our main algorithm to saddle point problems.

\subsection{Extension to multivariate monotone inclusions} Consider the following generalization of Problem~\ref{pro:main}.
\begin{pro}\label{pro:multi}
	{  Let $I$ and $K$ be finite subsets of $\N$}. For every $i \in I$ and every $k \in K$, let $\H_i$ and $\G_k$ be real Hilbert spaces. Set $\bm{\H} = \oplus_{i\in I} \H_i$ and $\bm{\G} = \oplus_{k \in K} \G_k$. For every $i \in I$ and every $k \in K$, let $A_i\colon \H_i\to 2^{\H_i}$ and $B_{k} \colon \G_k \to 2^{\G_k}$ be maximally monotone operators, let $L_{i,k} \colon \H_i \to \G_k$ be a bounded linear operator, let $D_i \colon \H_i \to \H_i$ be $\beta_i$-cocoercive for $\beta_i \in \RPP$, and let $\bm{C} \colon \bm{\H} \to \bm{\H} \colon \x \mapsto (C_i\x)_{i \in I}$, be $\zeta$-Lipschitz for $\zeta \in \RPP$. We aim to solve the system of primal inclusions
	\begin{align}\label{eq:multipri}
		\text{ find } &\x \in \bm{\H} \text{ such that} \nonumber\\
		&(\forall i \in I) \quad  0 \in A_i x_i +\sum_{k \in K} L^*_{i,k}B_k\left(\sum_{j \in I}L_{j,k}x_j\right)  + D_i x_i + C_i \x,  
	\end{align} 
	together with the associated system of dual inclusions
	\begin{align}\label{eq:multidual}
		\text{ find }& \bm{u} \in\ \bm{\G}
		\text{ such that} \nonumber \\
		& (\exists \x \in \bm{\H})\ (\forall i \in I)\ (\forall k \in K)\quad  
		\begin{cases}
			-\sum_{k \in K} L^*_{i,k}u_{k} \in A_i x_i + D_i x_i + C_i\x,\\
			u_{k} \in B_{k}\left(\sum_{j \in I}L_{j,k}x_j\right),
		\end{cases}
	\end{align}
	under the assumption that its primal-dual solution set $\widetilde{\bm{Z}}$ is not empty.
\end{pro}
This multivariate inclusion problem has been studied in \cite{AttouchBricenoCombettes2010,CombettesMinh2022,Comb13,CombettesEckstein2018MP} additionally considering parallel sums. In particular, the authors in \cite{CombettesMinh2022} proposed an asynchronous block-iterative algorithm which {  does not} need to {  evaluate} all the operators at each iteration. {  In this section, using product spaces techniques, we propose an algorithm for solving Problem~\ref{pro:multi} based on Algorithm~\ref{algo:algo1}. This algorithm requires evaluating all the operators at each iteration. However, it involves only one {  evaluation} of $L_{i,k}$ and $L^*_{i,k}$ at each iteration.}
\begin{cor}\label{cor:several2}
	In the context of Problem~\ref{pro:multi}, let $\varepsilon 
	\in 
	\,]0,1[$,  let $(\bm{x}_0,\bm{u}_0)\in 
	\bm{\H}\times \bm{\G}$, and let $(\sigma,\tau)\in \RPP^2$. Consider the sequence $\{(\bm{x}_n,\bm{u}_n)\}_{n \in \N} \subset \bm{\H}\times \bm{\G}$ defined recursively by 
	\begin{equation}\label{eq:algogen}
		(\forall n\in\N)\quad 
		\begin{array}{l}
			\left\lfloor
			\begin{array}{l}
				\textnormal{for every  } i \in I\\
				p_{i,n+1} = C_i \bm{x}_n\\
				z_{i,n+1} = J_{\tau A_i} \left(x_{i,n}-\tau \left(\sum_{k \in K}L_{i,k}^* 
				u_{k,n}+p_{i,n+1}+D_i x_{i,n}\right)\right) \\
				q_{i,n+1} = \tau(C_i \bm{z}_{n+1}-p_{i,n+1})\\
				\left\lfloor
				\begin{array}{l}
					\textnormal{for every } k \in K\\
					u_{k,n+1} =  J_{\sigma
						B_{k}^{-1}}\left(u_{k,n}+\sigma \sum_{j \in K}L_{j,k}( 
					2z_{j,n+1}-x_{j,n}-q_{j,n+1})\right)\\
				\end{array}
				\right.\\	x_{i,n+1} = z_{i,n+1}-q_{i,n+1}.
			\end{array}
			\right.
		\end{array}
	\end{equation}
	Set $\beta=\min\{\beta_1,\ldots,\beta_m\}$ and $\ell = \sum_{k \in K}(\sum_{i\in  I}\|L_{i,k}\|)^2$. Suppose that	
	\begin{equation}
		\tau \leq 2\beta\varepsilon
	\end{equation}	
	and that
	\begin{equation}
		\tau \sigma \ell +\tau^2\zeta^2< 1-\varepsilon.
	\end{equation}
	Then, there exists $(\bm{x},\bm{u})\in \widetilde{\bm{Z}}$ such that $(\bm{x}_n,\bm{u}_n)\weak (\bm{x},\bm{u})$.
\end{cor}
\begin{proof}
Consider the operators
\begin{align*}
	&\bm{A} \colon  \bm{\H} \to 2^{\bm{\H}} \colon \bm{x} \mapsto  \underset{i \in I}{\textnormal{\huge$\times$}}	A_ix_i,\quad 
	\bm{B} \colon  \bm{\G} \to 2^{\bm{\G}} \colon \bm{u} \mapsto  \underset{k \in K}{\textnormal{\huge$\times$}}	B_{k}u_k,\\
	&\bm{D}\colon  \bm{\H} \to \bm{\H} \colon \bm{x} \mapsto (D_ix_i)_{i \in I}, \text{ and }
	\bm{L} \colon  \bm{\H} \to \bm{\G} \colon \bm{x} \mapsto \left(\sum_{i \in I}L_{i,k}x_{i}\right)_{k \in K}.
\end{align*}
Note that
\begin{equation*}
	\bm{L}^* \colon  \bm{\G} \to \bm{\H} \colon \bm{u} \mapsto \left(\sum_{k \in K}L^*_{i,k}u_{k}\right)_{i \in I}.
\end{equation*}
Hence, \eqref{eq:multipri} can be written as
\begin{equation}
	\text{ find } \x \in \bm{\H} 
	\text{ such that } \  0 \in \bm{Ax+ L^*BL +Cx+Dx}.
\end{equation}
Furthermore, $\bm{C}$ is a $\beta$-cocoercive operator and it follows from \cite[Proposition 
20.23]{bauschkebook2017} that $\bm{A}$ and $\bm{B}$ are maximally monotone operators. Then, Problem~\ref{pro:multi} is a particular instance of Problem~\ref{pro:main}.	Additionally, it follows from 
\cite[Proposition~23.18]{bauschkebook2017} that
\begin{equation}
	J_{\tau \bm{A}} = (J_{\tau A_i})_{i \in I} \textnormal{ and }	J_{\sigma \bm{B}^{-1}} = (J_{\sigma B_k^{-1}})_{k \in K}.
\end{equation}	
Therefore, in this setting, the recurrence in \eqref{eq:algogen} correspond to Algorithm~\ref{algo:algo1}.
Furthermore, $\bm{L}$ is a bounded linear operator such that
\begin{equation}
	\|\bm{L}\x\|^2= \sum_{k\in K} \left(\sum_{i \in I}\|L_{i,k}x_i\|\right)^2\leq  \sum_{k\in K}\left(\sum_{i \in I}\|L_{i,k}\|\right)^2\|\x\|^2 = \ell\|\x\|^2,
\end{equation}
then, the result follows from Theorem~\ref{teo:conve1}.
\end{proof}
\begin{rem}
	Problem~\ref{pro:multi} models multivariate minimization problems and its applications as multicomponent image recovery problems, coupling evolution inclusions, network flows, among others (see \cite{AttouchBricenoCombettes2010,BricenoCombettes2009,Briceno2011ImRe}). 	%and saddle-point problems arising in generative adversarial networks \cite{Rieger2020AMC,RyuYuanYinGAN2019}. 
	Consider the following optimization problem
	\begin{equation}\label{eq:multipoti}
		\min_{\x\in \bm{\H}} \sum_{i \in I} \left( f_i(x_i)+\sum_{k \in K} g_k\left(\sum_{j \in I} L_{j,k} x_j\right)+d_i(x_i)\right)
		+\bm{h}(\x),
	\end{equation} 
	where for every $i \in I$ and every $k \in K$, $f_i\in \Gamma_0(\H)$, $g_i \in \Gamma_0(\G)$, $L_{i,k} \colon \H_i \to \G_i$ is a 
	bounded linear operator, $d_i$ is a convex differentiable function with  $1/\beta_i$-Lipschitz gradient for $\beta_i \in \RPP$, and $\bm{h} \colon \bm{\H} \to \R$ is a convex differentiable function with 
	$\zeta$-Lipschitz gradient for $\zeta \in \RPP$. Under adequate qualification conditions (see for instance \cite[Proposition 5.3]{Comb13}) the optimality condition for the optimization problem in \eqref{eq:multipoti} is
	\begin{equation}\label{eq:inclusaddle}
		(\forall i \in I) \quad 0
		\in \partial f_i (x_i) +\sum_{k \in K}L^*_{i,k}\left(\partial g_i \left(\sum_{j \in I}L_{j,k} x_j\right) \right)+ \nabla d_i (x_i)  + \nabla_i \bm{h}({\x}).
	\end{equation}
	By considering $A_i=\partial f_i$, $B_i=\partial g_i$, $C_i=\nabla_i \bm{h}$, and $D_i=\nabla d_i$, \eqref{eq:inclusaddle} is a particular instance of \eqref{eq:multipri}. Therefore, the optimization problem in \eqref{eq:multipoti} can be solved by the recurrence in \eqref{eq:algogen} which, in this setting, can be written as		
	\begin{equation}\label{eq:algogenopti}
		(\forall n\in\N)\  
		\begin{array}{l}
			\left\lfloor
			\begin{array}{l}
				\textnormal{for every  } i \in I\\
				p_{i,n+1} = \nabla_i \bm{h} (\bm{x}_n)\\
				z_{i,n+1} = \prox_{\tau f_i} \left(x_{i,n}-\tau \left(\sum_{k \in K}L_{i,k}^* 
				u_{k,n}+p_{i,n+1}+\nabla h_i (x_{i,n})\right)\right) \\
				q_{i,n+1} = \tau(\nabla_i \bm{h} (\bm{z}_{n+1})-p_{i,n+1})\\
				\left\lfloor
				\begin{array}{l}
					\textnormal{for every } k \in K\\
					u_{k,n+1} =  \prox_{\sigma
						g^*_{k}}\left(u_{k,n}+\sigma \sum_{j \in K}L_{j,k}( 
					2z_{j,n+1}-x_{j,n}-q_{j,n+1})\right)\\
				\end{array}
				\right.\\	x_{i,n+1} = z_{i,n+1}-q_{i,n+1}.
			\end{array}
			\right.
		\end{array}
	\end{equation}
\end{rem}

\subsection{Saddle point problems}\label{section:saddle} Several applications in machine learning, for example generative adversarial networks \cite{RyuYuanYinGAN2019} (see also \cite{Hamedani2021} and the references therein), are modeled by saddle point problems. In this subsection we apply our algorithm to the following saddle point problem:
\begin{equation}\label{eq:saddle}
	\min_{x \in \H_1}\max_{y\in \H_2} f_1(x)+f_2(L_1x)+f_3(x)+\Psi(x,y)-g_1(y)-g_2(L_2y)-g_3(y),	
\end{equation}	
where $\H_1,\H_2,\G_1,\G_2$ are real Hilbert spaces, $f_1 \in \Gamma_0(\H_1)$, $g_1 \in \Gamma_0(\H_2)$, $f_2 \in \Gamma_0(\G_1)$, $f_2 \in \Gamma_0(\G_2)$, $L_1\colon \H_1\to \G_1$ and $L_2\colon \H_2\to \G_2$ are bounded linear operators, $f_3 \colon \H_1 \to \R$ and $g_3 \colon \H_2 \to \R$ are convex differentiable functions with $(1/\beta_1)$ and $(1/\beta_2)$ Lipschitz gradient, respectively, and $\Psi \colon  \H_1\times  \H_2 \to \R$ is a convex-concave function with $\zeta$-Lipschitz gradient. Under standard qualification conditions, the optimality condition for a solution $(x,y) \in \H_1\times\H_2$ to \eqref{eq:saddle} is
\begin{equation}\label{eq:optisaddle}
	\begin{cases}
		0 \in \partial f_1 (x) + L_1^*\partial f_2(L_1x) + \nabla f_3 (x)+\nabla_x \Psi (x,y)\\
		0 \in \partial g_1 (y) + L_2^*\partial g_2(L_2y) + \nabla g_3 (y)-\nabla_y \Psi (x,y).
	\end{cases}
\end{equation}	
Additionally, we can write \eqref{eq:optisaddle} as the primal inclusion in \eqref{eq:primalinclu}, by defining 
\begin{align*}
	&A = \partial f_1 \otimes \partial g_1,
	\quad 
	B = \partial f_2 \otimes \partial g_2, \quad  L = (L_1,L_2)\\
	&D = (\nabla f_3, \nabla g_3 ),\quad  C = (\nabla_x \Psi, -\nabla_y \Psi).
\end{align*}
%
%\begin{align*}
%		&A \colon  {\H}_1\times \H_2 \to 2^{\H_1\times\H_2} \colon (x,y) \mapsto  \partial f_1 (x)\times \partial g_1 (y)
%		\quad 
%		B \colon  {\G}_1\times \G_2 \to 2^{\G_1\times\G_2} \colon (u,v) \mapsto  \partial f_2 (u)\times \partial g_2 (v)\\
%		&L \colon \H_1\times\H_2 \to \G_1\times \G_2 \colon (x,u) \mapsto (L_1x,L_2y)\\
%		&D\colon  \H_1\times \H_2 \to \H_1\times\H_2 \colon (x,y) \mapsto (\nabla f_3 (x), \nabla g_3 (y)) \text{ and }
%		&C \colon  \H_1\times \H_2 \to \H_1\times\H_2 \colon (x,y) \mapsto (\nabla_x \psi (x,y), \nabla_y \psi (x,y)).
%	\end{align*}
Hence, by applying Algorithm~\ref{eq:algo1} to \eqref{eq:optisaddle}, we derive the following recurrence to solve the optimization problem in \eqref{eq:saddle}.
\begin{equation}
(\forall n\in\N)\quad 
\begin{array}{l}
	\left\lfloor
	\begin{array}{l}
		p^1_{n+1} = \nabla_x \Psi (x_n)\\
		p^2_{n+1} = \nabla_y \Psi (y_n)\\
		z^1_{n+1} = \prox_{\tau f_1} (x_n-\tau (L_1^* u_n+p^1_{n+1}+\nabla f_3 
		(x_n))) \\
		z^2_{n+1} = \prox_{\tau g_1} (y_n-\tau (L_2^* v_n+p^2_{n+1}+\nabla g_3 
		(y_n))) \\
		q^1_{n+1} = \tau(\nabla_x \Psi(z^1_{n+1})-p^1_{n+1})\\
		q^2_{n+1} = \tau(\nabla_y \Psi(z^2_{n+1})-p^2_{n+1})\\
		u_{n+1} =  \prox_{\sigma
			f_2^*}\left(u_{n}+\sigma L_1( 
		2z^1_{n+1}-x_{n}-q^1_{n+1})\right)\\
		v_{n+1} =  \prox_{\sigma
			g_2^*}\left(v_{n}+\sigma L_2( 
		2z^2_{n+1}-y_{n}-q^2_{n+1})\right)\\
		x_{n+1} = z^1_{n+1}-q^1_{n+1}\\
		y_{n+1} = z^2_{n+1}-q^2_{n+1},
	\end{array}
	\right.
\end{array}
\end{equation}
where $(x_0,y_0)\in \H_1\times \H_2$ and $(u_0,v_0)\in \G_1\times \G_2$, and the convergence is guaranteed for $(\tau,\sigma) \in \RPP$ satisfying \eqref{eq:steps1} and \eqref{eq:steps2} for $\varepsilon \in ]0,1[$ and $\beta = \min\{\beta_1,\beta_2\}$. 
{ \begin{rem}
	Note that in this case, $C$ is Lipschitz but not necessarily a cocoercive operator; for example, $C$ could be a linear skew operator. 
\end{rem}	
}

\section{Numerical Implementation}\label{sec:exp}
{  In this section, we present numerical experiments exhibiting the 
performance of Algorithm~\ref{algo:algo1} in an optimization context. We compare Algorithm~\ref{algo:algo1} (FPDHF) with the methods proposed in \cite[Corollary~5.2]{BricenoDavis2018} (FBHFPD) and in \cite[Section~6]{MorinBanertGiselsson2022} (PDBTR and PDRCK) for
solving Problem~\ref{pro:main}. This section is divided into three parts. First, we introduce the optimization problem to be solved. Next, we test the algorithms with different parameter settings on the optimization problem. Finally, we demonstrate that FPDHF allows larger step sizes than PDBTR and PDRCK, leading to numerical advantages.}

The numerical experiments were performed on a laptop with AMD Ryzen 5 3550Hz, Radeon Vega Mobile Gfx, 
and 32 Gb RAM.

\subsection{Numerical experiments on image deblurring problems}\label{sec:subsecNE1}
{ To compare the proposed method with those mentioned earlier, we present the following image deblurring model. Our focus is on optimization efficiency rather than the deblurring process itself; therefore, we do not compare our method with other approaches, such as those based on convolutional neural networks (see \cite{Wu2021NNID} and references therein).}

Let $(N,M) \in \N^2$ and let $x \in \R^{N\times N}$ be an image to be recovered from 
an observation
\begin{equation}\label{eq:modelim}
z = Tx+\epsilon,
\end{equation}	
where { $T\colon \R^{N\times N}\times \R^{M\times M}$} is an operator representing a blur process and $\epsilon$ 
is an additive noise perturbing the observation. We aim to recover 
the image $x$ by solving the following optimization problem.
\begin{equation}\label{eq:opex}
\min_{x \in [0,x_{\textnormal{max}}]^{N\times N}} 
\frac{1}{2}\|Tx-z\|^2+\lambda_1\|\nabla 
x\|_1+\lambda_2 H_\delta(Wx),
\end{equation}
where $(\lambda_1,\lambda_2) \in \RPP^2$ {  are regularization parameters}, $\|\cdot\|_1$ is the 
$\ell^1$ norm, and, given $\delta >0$, $H_\delta:\R^{N\times N}\to \R$ is the Huber function defined, for all $x = (x_{i,j})_{1\leq i,j\leq N} \in \R^{N\times N}$, by $H_{\delta}(x) 
=  \sum_{j=1}^N\sum_{i=1}^N h_{\delta}(x_{i,j})$, where
\begin{equation}
(\forall \eta \in \R) \quad
h_{\delta} (\eta) = \begin{cases}
	|\eta|-\frac{\delta}{2}, & \text{ if } |\eta| > \delta,\\
	\frac{\eta^2}{2\delta}, & \text{ otherwise}.
\end{cases}
\end{equation}
The linear operator $\nabla\colon x\mapsto (D_1x,D_2x)$ is the 
discrete gradient, where \linebreak $D_1\colon\R^{N\times N}\to\R^{N\times (N-1)}$ represents the 
horizontal differences and $D_2\colon\R^{N\times N}\to\R^{(N-1)\times N}$ represents the vertical differences of a matrix. Its adjoint $\nabla^*$ is 
the discrete divergence \cite{TV-chambolle}. The 
operator $W:\R^{N\times N}\to \R^{N\times N}$ is an orthogonal basis wavelet 
transform { (we assume that $N$ is a power of 2)}. In particular, we choose a Haar basis. The constraint $x \in [0,x_{\textnormal{max}}]^{N\times N}$ impose {  pixel by pixel}
the range of the image to be restored.  The term 
$\frac{1}{2}\|Tx-z\|^2$ measure the difference between $Tx$ and $z$. The term 
$\lambda_1\|\nabla x\|_1+\lambda_2H_\delta(Wx)$ {  controls both  the energy of the image and its frequency in the wavelet domain}. This total 
variation \cite{ROF1992} and wavelet penalization for image 
restoration has 
been studied for example in 
\cite{Chambolle2004l1unified,Yin2014TV+W,Yin2008TV+W,Yin2010TVl1l2}.
By setting $f=\iota_{[0,x_{\textnormal{max}}]^{N\times N}}$, 
$g=\lambda_1\|\cdot\|_1$, $L=\nabla$, $h =\lambda_2H_\delta\circ W$, 
and $d=\frac{1}{2}\|T\cdot-z\|^2$, the optimization problem in 
equation 
\eqref{eq:opex} is a particular instance of Problem~\ref{pro:opti} and therefore it can be solved by the recurrence in \eqref{eq:algo1opti}. 
{ We approximate $\|L\|$ as $\sqrt{8}$ (see \cite{Chambolle2004}). Note that,} $\nabla d =T^*(T\cdot-z) $ is 
$\|T\|^2$-Lipschitz, and $\nabla h$ is 
$\lambda_2/\delta$-Lipschitz (for an explicit expression of $\nabla H_\delta$ see \cite{BricenoPustelnik2023}). The blur operator $T$
is set as a Gaussian blur of size $9\times9$ and standard deviation 4 (by { the} MATLAB function fspecial), thus, we have $\|T\|=1$ { (calculated in MATLAB using the functions provided in \cite[VIP~11]{Hansen2006})}. The observation $z$ is obtained through \eqref{eq:modelim} where $\epsilon$ is an additive zero-mean white Gaussian noise with standard deviation $10^{-3}$ (implemented by the imnoise function in MATLAB). 

\subsection{Testing the parameters}
In this section, we compare the algorithms for different values of $\lambda_1$, $\lambda_2$, $\delta$, Haar basis level, and $N$. The step-sizes used in the algorithms are detailed in Table~\ref{T:stepsize}. Specifically, we set $\tau = 0.5$ and choose $\sigma$ as large as possible ensuring their respective convergence condition (see \cite[Eq.~ (5.18)]{BricenoDavis2018} for FBHFPD and  \cite[Corollary~6.1\&6.3]{MorinBanertGiselsson2022} for PDBTR and PDRCK). As a stopping criterion, we use the relative primal-dual error, defined as $((\|x_{n+1}-x_{n}\|^2+\|u_{n+1}-u_{n}\|^2)/(\|x_{n}\|^2+\|u_{n}\|^2))^{1/2}$, with a tolerance of $10^{-6}$. 

First, we consider $N=128$, $\delta = 10^{-2}$, Haar basis of level 3, and different values for $\lambda_1$ and $\lambda_2$. The results including Iteration Number (IN), CPU Time (T) in seconds, and Objective Value (OV) are presented in Table~\ref{T:varlam128}. From this table, we observe that while PDBTR generally performs best in terms of iteration number and time, FPDHF is competitive and even overcome PDBTR in iteration number in some cases. However, PDBTR requires less time per iteration than FPDHF, this can be explained by the fact that FPDHF involves an additional evaluation of the Lipschitz operator at each iteration. Additionally, the table shows similar results when $\lambda_2$ remains constant and $\lambda_1$ varies. On the other hand, when $\lambda_1$ is constant and $\lambda_2$ decreases, PDRCK becomes more competitive, while the other algorithms maintain their performance with minor variations. In all cases, FBHFPD needs more iterations and CPU time than the other algorithms. Regarding objective value, all the algorithms converge to the same value. Figure~\ref{fig:graphcomparison_3} illustrates the relative error along iterations for $\lambda_1 \in \{10^{-1},10^{-2},10^{-3}\}$ and $\lambda_2 = 10^{-3}$. The original image, the observation, and recovered images are shown in Table~\ref{tab:rec_128}. In the same setting, we consider $\delta \in\{1,100\}$ and Haar basis of levels $1$ and $5$. The results, shown in Table~\ref{T:vardelta128}, are consistent with the previous analysis. Finally, Table~\ref{T:stepsize256} presents the results for $N=256$. Here, when $\lambda_1=10^{-1}$ and $\lambda_2=10^{-3}$,  FBHFPD is more competitive in iteration number, but it requires considerably more CPU time than PDBTR, as it involves additional evaluations of the Lipschitz and linear operators. FPDHF remains competitive with PDBTR in iteration number but need more CPU time; this difference is more pronounced when $\lambda_1 = 10^{-1}$ and diminishes when $\lambda_1 \in \{10^{-2},10^{-3}\}$.
The relative error along iterations and the original image, the observation, and recovered images are shown in Figure~\ref{fig:graphcomparison_2} and in Table~\ref{fig:graphcomparison_2}, respectively. Despite the advantages of PDBTR, the next section provides examples where FPDHF achieves better results than PDBTR because it allows larger step sizes.
\begin{center}
\begin{table}\centering \resizebox{12cm}{!}{
		\begin{tabular}{|c|c|c|c|}\hline
			Algorithm & $\tau$ & $\sigma$ & \\ \hline 
			FPDHF &  0.5 & $0.9999(1-\varepsilon-\tau^2\zeta^2)/(\|L\|^2\tau) $& $\varepsilon = \tau/2\beta$ \\
			\hline
			FBHFPD &  0.5 & $
			\frac{\sigma+\tau-\sqrt{(\sigma-\tau)^2+4\sigma^2\tau^2\|L\|^2}}{2\sigma\tau}> \frac{1}{4\beta} + \sqrt{\frac{1}{16\beta^2}+\zeta^2}$ & $\theta =1 $\\
			\hline 
			PDBTR & 0.5 & $0.9999(1-\tau/(2\beta)-2\tau\zeta)/(\|L\|^2\tau)$ & $\lambda_n \equiv 2$\\
			\hline
			BDRCK &  0.5 &$0.9999(1-\tau/(2\beta)-2\tau\zeta)/(2\|L\|^2\tau)$& - \\
			\hline
	\end{tabular}}	\caption{Step-sizes implemented in the experiments. $\zeta = \lambda_2/\delta $ and $\beta = 1$.}\label{T:stepsize}
\end{table}
\end{center}

\begin{center}
\begin{table}\centering \resizebox{12cm}{!}{
		\begin{tabular}{|c|c|c|c|c|c|c|c|c|c|}
			\cline{2-10}  \multicolumn{1}{ }{}&\multicolumn{3}{|c|}{$N=128$}&\multicolumn{3}{|c|}{Haar Basis Level 3 }&\multicolumn{3}{|c|}{$\delta = 10^{-2}$} \\
			\cline{2-10}
			\multicolumn{1}{ c|}{}&I.N. & T. (s) & O.V. & I.N. & T. (s) & O.V. & I.N. & T. (s) & O.V. \\ 
			\hhline{-=========} 
			Algorithm &\multicolumn{3}{|c|}{$\lambda_1=10^{-1}$, $\lambda_2=10^{-3}$}&\multicolumn{3}{|c|}{$\lambda_1=10^{-2}$, $\lambda_2=10^{-3}$}&\multicolumn{3}{|c|}{$\lambda_1=10^{-3}$, $\lambda_2=10^{-3}$}\\ \hline
			FPDHF & 4232&18.5&28.24& 2511&10.2&4.34 & 3129& 12.5&1.48 \\ 
			\hline
			FBHFPD &6549&28.6&28.24 & 6646&29.0&4.34 & 7507& 32.6&1.48 \\
			\hline 
			PDBTR & 4689&15.7&28.24 & 2527&8.6&4.34& 3129 & 10.6 & 1.48 \\
			\hline
			BDRCK & 7808&27.9&28.24 & 2804&10.0&4.34 & 3129& 11.3 & 1.48 \\
			\hline\hline
			Algorithm &\multicolumn{3}{|c|}{$\lambda_1=10^{-1}$, $\lambda_2=10^{-4}$}&\multicolumn{3}{|c|}{$\lambda_1=10^{-2}$, $\lambda_2=10^{-4}$}&\multicolumn{3}{|c|}{$\lambda_1=10^{-3}$, $\lambda_2=10^{-4}$}\\ \hline 
			FPDHF & 4249&17.7&27.52& 4368&18.6&3.60 & 4307&18.0&0.77 \\ 
			\hline
			FBHFPD &6799&31.3&27.52 & 8311&38.1&3.60 & 9927&45.8&0.77 \\
			\hline 
			PDBTR &4291&15.6&27.52 & 4369&15.1&3.60
			& 4307&14.9&0.77\\
			\hline
			BDRCK &7173&26.7& 27.52 & 4464&16.4&3.60 & 4309 &   15.9&0.77 \\
			\hline\hline
			Algorithm &\multicolumn{3}{|c|}{$\lambda_1=10^{-1}$, $\lambda_2=10^{-5}$}&\multicolumn{3}{|c|}{$\lambda_1=10^{-2}$, $\lambda_2=10^{-5}$}&\multicolumn{3}{|c|}{$\lambda_1=10^{-3}$, $\lambda_2=10^{-5}$}\\ \hline 
			FPDHF & 4375&18.2&27.44 & 3826&15.9&3.52 & 4564& 19.2&0.70 \\ 
			\hline
			FBHFPD &6921&31.4&27.44 & 8365&38.7&3.52 & 10513&48.3&0.70 \\
			\hline 
			PDBTR & 4379&15.1&27.44& 3827&13.3&3.52 & 4564& 16.0& 0.70 \\
			\hline
			BDRCK & 7139&26.3&27.44 & 3887&14.6&3.52 & 4568&  17.1&0.70 \\
			\hline
	\end{tabular}}\caption{Comparison of FPDHF, 	FBHFPD, PDBTR, and BDRCK for different values of $\lambda_1$ and $\lambda_2$.}\label{T:varlam128}
\end{table}
\end{center}

\begin{figure}\label{fig:graph_1}
\subfloat[]{\label{fig:graph_11}
	\includegraphics[scale=0.322]{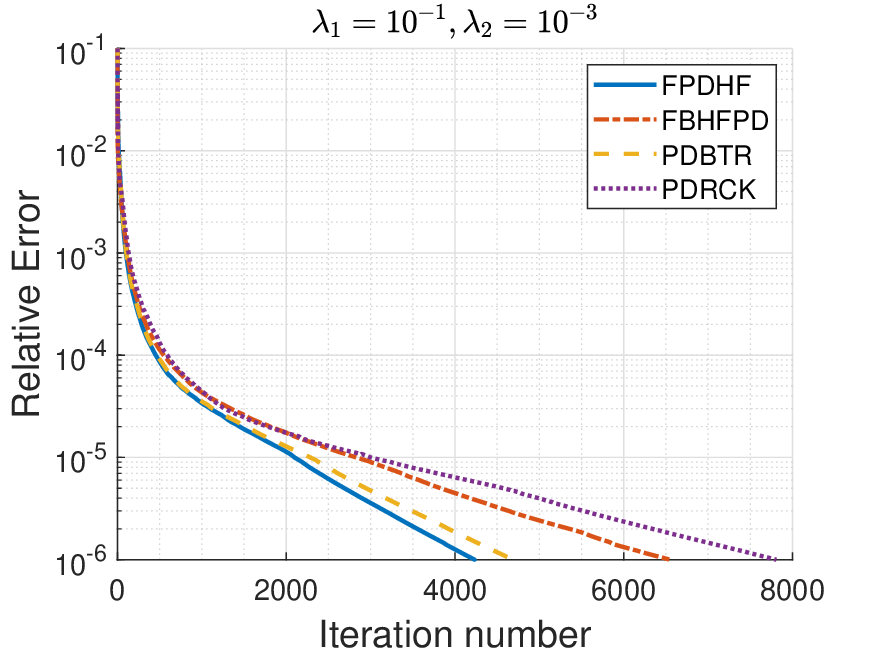}}
\subfloat[]{\label{fig:graph_12}
	\includegraphics[scale=0.322]{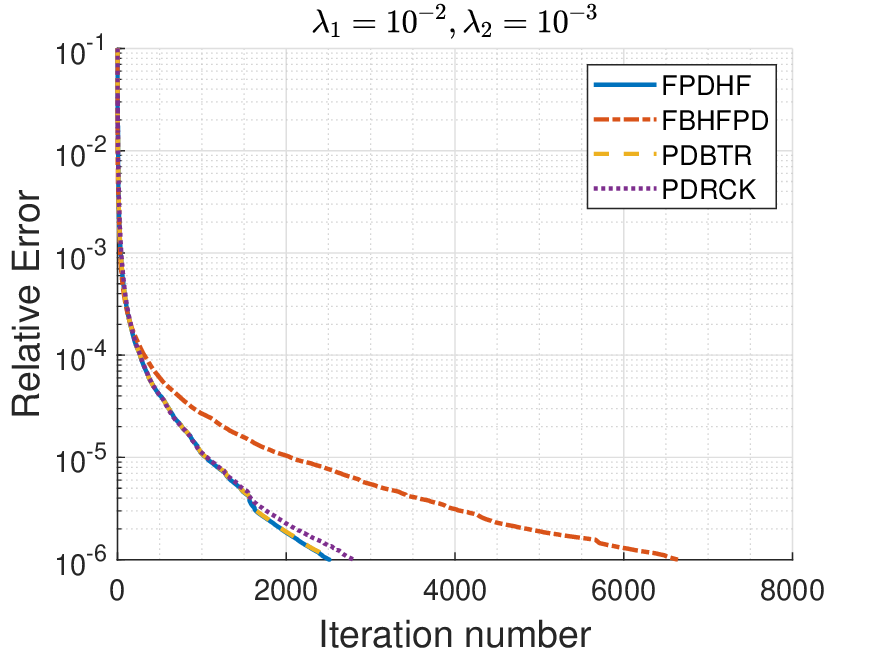}}
\subfloat[]{\label{fig:graph_13}
	\includegraphics[scale=0.322]{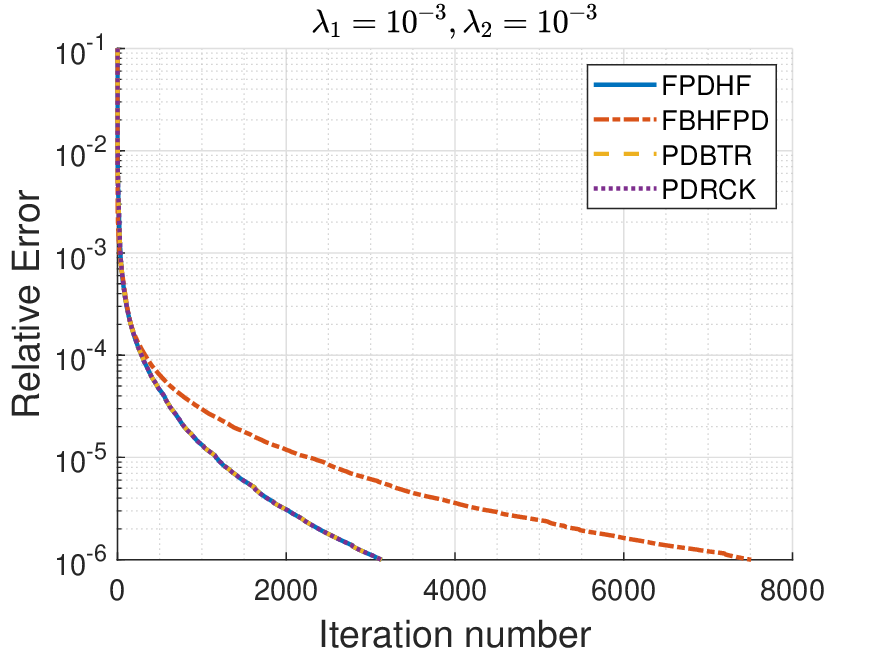}}
%	\subfloat[]{\label{fig:graph_13}
	%		\includegraphics[scale=0.325]{TIME_IT.eps}}
\caption{Relative error along iterations. $N=128$, $\delta=10^{-3}$, Haar basis level 3.} \label{fig:graphcomparison_3}
\end{figure}

{\centering
\begin{table}\centering \resizebox{12cm}{!}{
		\begin{tabular}{c|c|cccc}
			\multirow{7}{*}{\shortstack[l]{\vspace*{1cm}\\ \subfloat[Original]{\label{fig:recover_00}\includegraphics[scale=0.32]{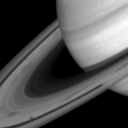}}\\ \subfloat[Blur and Noisy]{\label{fig:recover_01}\includegraphics[scale=0.32]{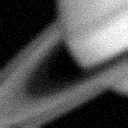}}}} & \multicolumn{1}{c|}{RP}& FPDHF & FBHFPD   & PDBTR & PDRCK\\ \cline{2-6}  & & & & &
			\\
			&\shortstack[l]{$\lambda_1=10^{-1}$\\ \vspace*{0.8cm}$\lambda_2=10^{-3}$}&\subfloat[ 25.62]{\label{fig:recover_113}\includegraphics[scale=0.32]{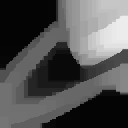}} & \subfloat[ 25.63]{\label{fig:recover_114}\includegraphics[scale=0.32]{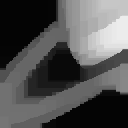}} &
			\subfloat[ 25.62]{\label{fig:recover_115}\includegraphics[scale=0.32]{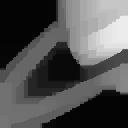}} &
			\subfloat[ 25.62]{\label{fig:recover_116}\includegraphics[scale=0.32]{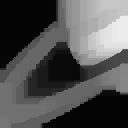}} \\ \cline{2-6} & & & & & \\
			&\shortstack[l]{$\lambda_1=10^{-2}$\\ \vspace*{0.8cm}$\lambda_2=10^{-3}$}&\subfloat[ 27.18]{\label{fig:recover_117}\includegraphics[scale=0.32]{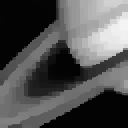}} & \subfloat[ 27.20]{\label{fig:recover_118}\includegraphics[scale=0.32]{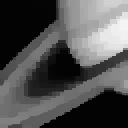}} &
			\subfloat[ 27.19]{\label{fig:recover_119}\includegraphics[scale=0.32]{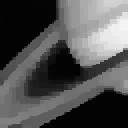}} &
			\subfloat[ 27.18]{\label{fig:recover_120}\includegraphics[scale=0.32]{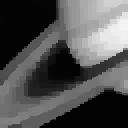}} \\ \cline{2-6}  & & & & & \\
			&\shortstack[l]{$\lambda_1=10^{-3}$\\ \vspace*{0.8cm}$\lambda_2=10^{-3}$}&\subfloat[ 27.01]{\label{fig:recover_121}\includegraphics[scale=0.32]{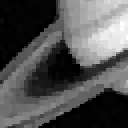}} & \subfloat[ 27.06]{\label{fig:recover_122}\includegraphics[scale=0.32]{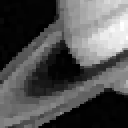}} &
			\subfloat[ 27.01]{\label{fig:recover_123}\includegraphics[scale=0.32]{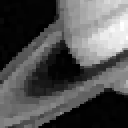}} &
			\subfloat[ 27.01]{\label{fig:recover_124}\includegraphics[scale=0.32]{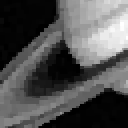}} 
	\end{tabular}}
	\caption{Original, blurred and noisy, and recovered images  for $N=128$, $\delta=10^{-2}$, Haar basis of level 3, and different regularization parameters (RP). The peak signal-to-noise ratio (dB) is displayed below each recovered image.}
	\label{tab:rec_128}
\end{table}
}
{
\begin{center}
	\begin{table}\centering \resizebox{12cm}{!}{
			\begin{tabular}{|c|c|c|c|c|c|c|c|c|c|}
				\cline{2-10}
				\multicolumn{1}{ c|}{} &\multicolumn{3}{|c|}{$\lambda_1=10^{-1}$, $\lambda_2=10^{-3}$}&\multicolumn{3}{|c|}{$\lambda_1=10^{-2}$, $\lambda_2=10^{-3}$}&\multicolumn{3}{|c|}{$\lambda_1=10^{-3}$, $\lambda_2=10^{-3}$}\\ 
				\cline{2-10} 
				\multicolumn{1}{ c|}{}&I.N. & T. (s) & O.V. & I.N. & T. (s) & O.V. & I.N. & T. (s) & O.V. \\ 
				\hhline{-=========} 
				Algorithm &\multicolumn{3}{|c|}{$N=128$}&\multicolumn{3}{|c|}{Haar Basis Level 3}&\multicolumn{3}{|c|}{$\delta=1$}\\ \hline 
				FPDHF & 4213&19.3& 28.01 & 2868&11.8&4.11 &3232&13.4&1.28 \\ 
				\hline
				FBHFPD & 6566& 28.9& 28.01&  7107&31.5&4.11 &8244&         36.6&1.28\\
				\hline 
				PDBTR &4217& 14.5& 28.01&2868&9.9&4.11&3232&11.2&1.28\\
				\hline
				BDRCK & 7092& 25.4& 28.01 & 2973&10.9&4.11 & 3227&  11.8&1.28\\
				\hline
				\hline
				Algorithm &\multicolumn{3}{|c|}{$N=128$}&\multicolumn{3}{|c|}{Haar Basis Level 3}&\multicolumn{3}{|c|}{$\delta=100$}\\ 
				\hline
				FPDHF &  4376&18.0&27.44& 3857&16.6&3.53 &4643&19.8&0.70 \\ 
				\hline
				FBHFPD & 6911&31.5&27.45&  8445&39.2&3.53 &10785&50.1&0.70\\
				\hline 
				PDBTR &4376&15.5&27.44& 3857&13.8&3.53& 4643&16.5&0.70\\
				\hline
				BDRCK & 7121&26.8&27.45&3924& 14.9&3.53&   4639&  17.5&0.70\\
				\hline
				\hline
				Algorithm &\multicolumn{3}{|c|}{$N=128$}&\multicolumn{3}{|c|}{Haar Basis Level 1}&\multicolumn{3}{|c|}{$\delta=10^{-2}$}\\ 
				\hline
				FPDHF &  4301& 17.5&30.07&3026&11.0&6.14&3932&15.2&3.29\\	\hline
				FBHFPD &6788&29.3&30.07&7278&30.8&6.14& 8835&37.5&3.29\\	\hline
				PDBTR &4770&16.2&30.07&3037&10.3&6.14&3932&13.2&3.29\\	\hline
				BDRCK &7909&28.6&30.07&3149&11.3&6.14&3932&14.1&3.29\\	
				\hline
				\hline
				Algorithm &\multicolumn{3}{|c|}{$N=128$}&\multicolumn{3}{|c|}{Haar Basis Level 5}&\multicolumn{3}{|c|}{$\delta=10^{-2}$}\\ 
				\hline
				FPDHF &4278&19.5&27.89&2571&11.7&4.01&3187&14.5&1.16\\ 
				\hline
				FBHFPD &6518&31.9&27.89&6567&32.3&4.01&7342&35.9&1.16\\
				\hline 
				PDBTR &4728&17.6&27.89&2587&9.6&4.01&3187&11.9&1.16\\
				\hline
				BDRCK& 7865&30.8&27.89&2869&11.3&4.01&3188&12.6&1.16\\
				\hline
		\end{tabular}}\caption{Comparison of FPDHF, 	FBHFPD, PDBTR, and BDRCK for different values of $\delta$ and Haar basis level.}\label{T:vardelta128}
	\end{table}
	\end{center}}
	
	\begin{center}
\begin{table}\centering \resizebox{12cm}{!}{
		\begin{tabular}{|c|c|c|c|c|c|c|c|c|c|}
			\cline{2-10}  \multicolumn{1}{ }{}&\multicolumn{3}{|c|}{$N=256$}&\multicolumn{3}{|c|}{Haar Basis Level 3 }&\multicolumn{3}{|c|}{$\delta = 10^{-2}$} \\
			\cline{2-10}
			\cline{2-10}
			\multicolumn{1}{ c|}{}&I.N. & T. (s) & O.V. & I.N. & T. (s) & O.V. & I.N. & T. (s) & O.V. \\ 
			\hhline{-=========} 
			Algorithm &\multicolumn{3}{|c|}{$\lambda_1=10^{-1}$, $\lambda_2=10^{-3}$}&\multicolumn{3}{|c|}{$\lambda_1=10^{-2}$, $\lambda_2=10^{-3}$}&\multicolumn{3}{|c|}{$\lambda_1=10^{-3}$, $\lambda_2=10^{-3}$}\\ \hline
			FPDHF & 10415&202.4&159.75 & 4174&83.7&22.85& 4212& 84.1&7.77\\ 
			\hline
			FBHFPD & 10965&224.7&159.75 & 9041&189.5&22.84 & 9295&      194.6&7.77 \\
			\hline 
			PDBTR & 10410&175.9&159.75& 4169&72.7&22.85&4212& 73.8&7.77\\
			\hline
			BDRCK & 11258&200.6&159.76 & 4147&75.1&22.85 &4212&         76.9&7.77 \\
			\hline
	\end{tabular}}	\caption{FPDHF, 	FBHFPD, PDBTR, and BDRCK for $N=256$.}\label{T:stepsize256}
\end{table}
\end{center}
\begin{figure}\label{fig:graph_2}
\subfloat[]{\label{fig:graph_21}
	\includegraphics[scale=0.322]{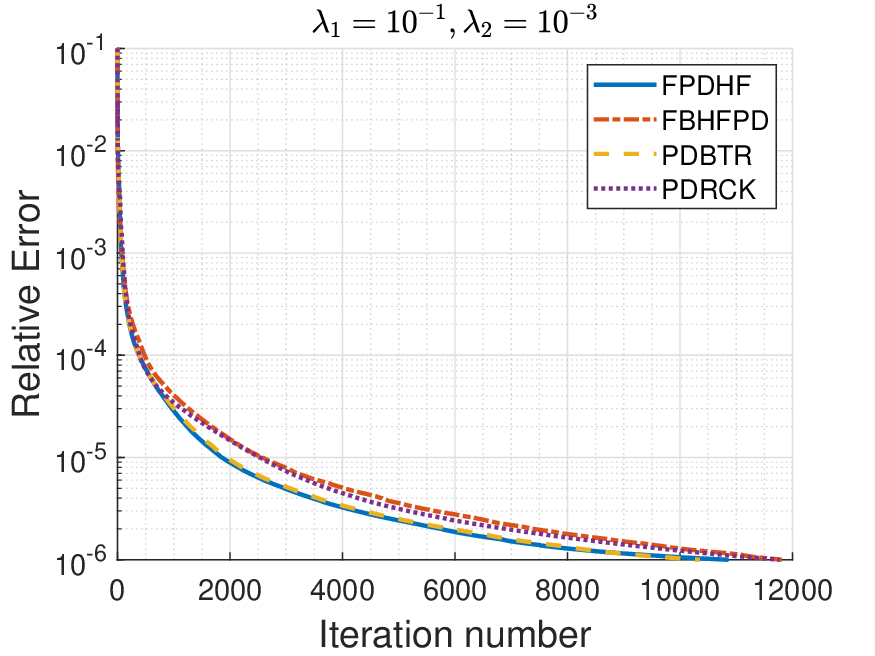}}
\subfloat[]{\label{fig:graph_22}
	\includegraphics[scale=0.322]{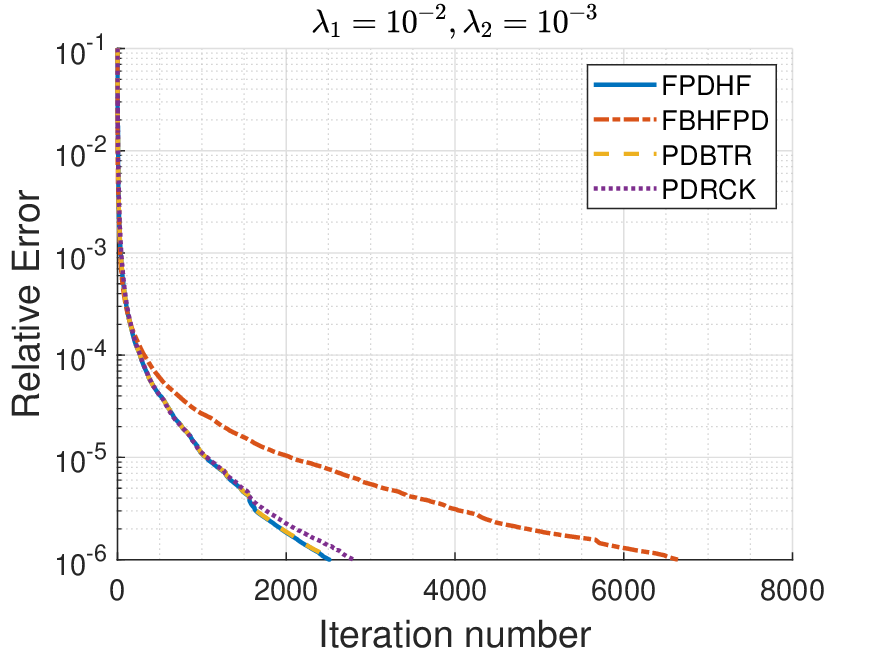}}
\subfloat[]{\label{fig:graph_23}
	\includegraphics[scale=0.322]{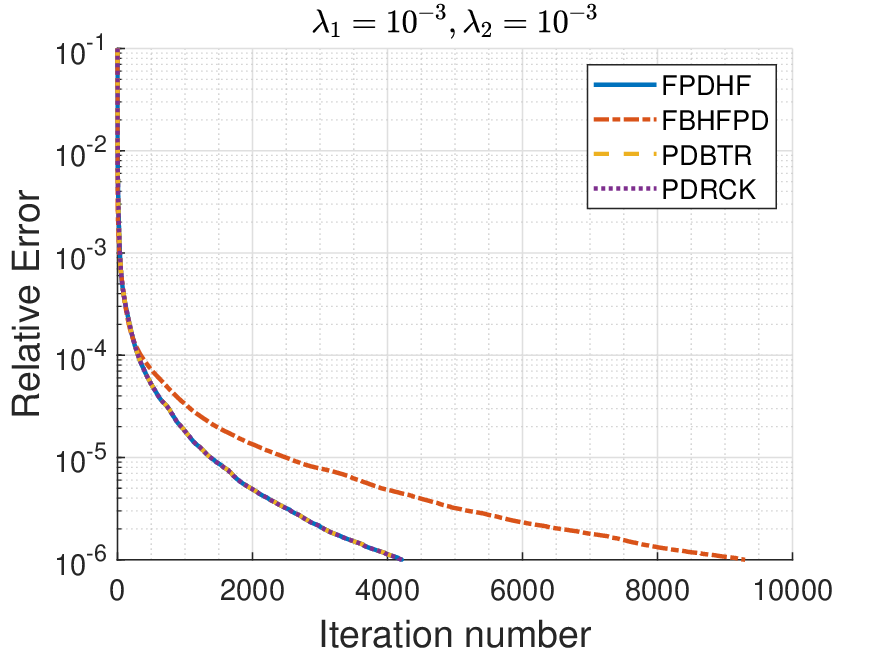}}
%	\subfloat[]{\label{fig:graph_13}
	%		\includegraphics[scale=0.325]{TIME_IT.eps}}
\caption{Relative error along iterations. $N=256$, $\delta=10^{-3}$, Haar basis level 3.} \label{fig:graphcomparison_2}
\end{figure}

{\centering
\begin{table}
	\begin{tabular}{c|c|cccc}
		\multirow{7}{*}{\shortstack[l]{\vspace*{1cm}\\ \subfloat[Original]{\label{fig:recover_02}\includegraphics[scale=0.25]{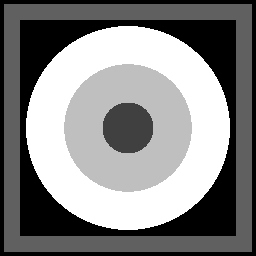}}\\ \subfloat[Blur and Noisy]{\label{fig:recover_03}\includegraphics[scale=0.25]{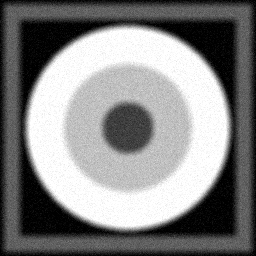}}}} & \multicolumn{1}{c|}{RP}& FPDHF  & FBHFPD  & PDBTR & PDRCK\\ \cline{2-6}  & & & & &
		\\
		&\shortstack[l]{$\lambda_1=10^{-1}$\\ \vspace*{0.8cm}$\lambda_2=10^{-3}$}&\subfloat[23.44]{\label{fig:recover_213}\includegraphics[scale=0.25]{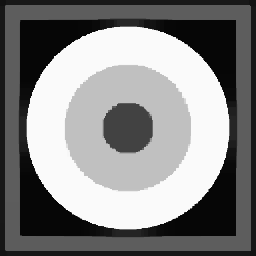}} & \subfloat[23.44]{\label{fig:recover_214}\includegraphics[scale=0.25]{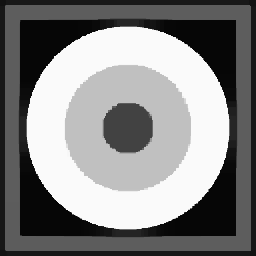}} &
		\subfloat[23.44]{\label{fig:recover_215}\includegraphics[scale=0.25]{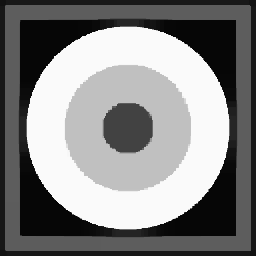}} &
		\subfloat[23.43]{\label{fig:recover_216}\includegraphics[scale=0.25]{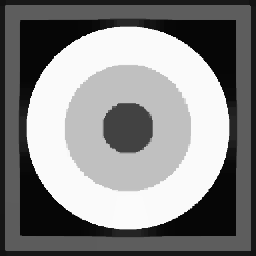}} \\ \cline{2-6} & & & & & \\
		&\shortstack[l]{$\lambda_1=10^{-2}$\\ \vspace*{0.8cm}$\lambda_2=10^{-3}$}&\subfloat[27.21]{\label{fig:recover_217}\includegraphics[scale=0.25]{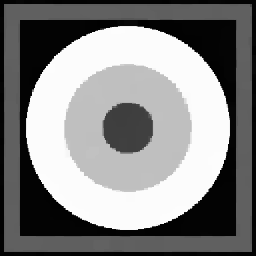}} & \subfloat[27.27]{\label{fig:recover_218}\includegraphics[scale=0.25]{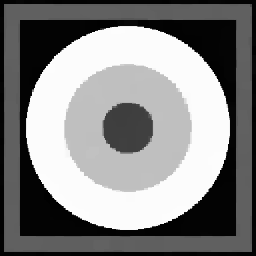}} &
		\subfloat[27.21]{\label{fig:recover_219}\includegraphics[scale=0.25]{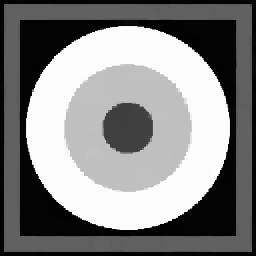}} &
		\subfloat[27.21]{\label{fig:recover_220}\includegraphics[scale=0.25]{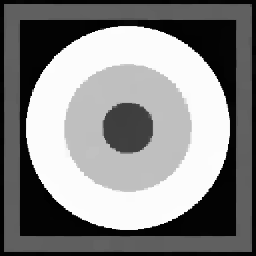}} \\ \cline{2-6}  & & & & & \\
		&\shortstack[l]{$\lambda_1=10^{-3}$\\ \vspace*{0.8cm}$\lambda_2=10^{-3}$}&\subfloat[ 26.03]{\label{fig:recover_221}\includegraphics[scale=0.25]{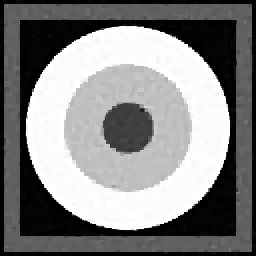}} & \subfloat[26.11]{\label{fig:recover_222}\includegraphics[scale=0.25]{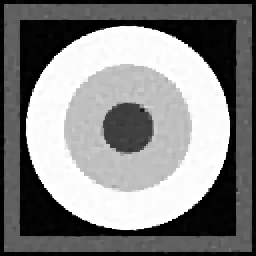}} &
		\subfloat[26.03]{\label{fig:recover_223}\includegraphics[scale=0.25]{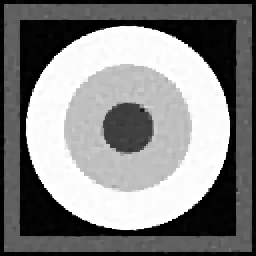}} &
		\subfloat[26.03]{\label{fig:recover_224}\includegraphics[scale=0.25]{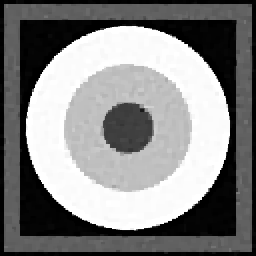}} 
	\end{tabular}
	\caption{Original, blurred and noisy, and recovered images  for $N=256$, $\delta=10^{-2}$, Haar basis of level 3, and different regularization parameters (RP). The peak signal-to-noise ratio (dB) is displayed below each recovered image.}
	\label{tab:rec_256}
\end{table}
}

\subsection{Comments on the step-size}
The choice of step-sizes directly influences the speed of convergence in primal-dual algorithms. For instance, selecting the largest possible step-sizes improves the numerical convergence of the method \cite{BricenoRoldan2021,HeYuan2012}. On the other hand, \cite{ChambolleAdaptive24} proposed an adaptive strategy to select the step-size of the Stochastic Primal–Dual Hybrid Gradient algorithm to accelerate its numerical convergence, highlighting the importance of appropriate step-size selection. In this section, we demonstrate that FPDHF allows for larger step-sizes than PDBTR and PDRCK, resulting in numerical advantages.  

In the context of Problem~\ref{pro:main}, let $(\tau,\sigma) \in \RPP^2$. By \cite[Corollary~6.1]{MorinBanertGiselsson2022}, the convergence of PDBTR is guaranteed for $(\tau,\sigma)$ if there exist $\epsilon >0$ such that 
\begin{equation}\label{eq:stepsize2}
\tau \sigma \|L\|^2 +(|2-\lambda_n|+|2-\lambda_{n+1}|)\sqrt{\tau\sigma}\|L\|+\tau\left(2\zeta + \frac{1}{2\beta}\right) < 1-\epsilon,
\end{equation}
where, for every $n \in \N$, $\lambda_n \in \R$. If \eqref{eq:stepsize2} holds, in particular we have
\begin{equation}\label{eq:stepsize25}
\tau \sigma \|L\|^2 +2\tau\zeta + \frac{\tau}{2\beta}< 1.
\end{equation} 
Moreover, it follows from \eqref{eq:stepsize25} that
\begin{equation*}
2\tau\zeta < 1 \Leftrightarrow \tau^2\zeta^2 < \frac{\tau\zeta}{2} \Rightarrow \tau^2\zeta^2 < 2\tau\zeta.
\end{equation*}
Hence
\begin{equation}
\sigma \|L\|^2 +\tau^2\zeta^2 + \frac{\tau}{2\beta} 	< \tau \sigma \|L\|^2 +2\tau\zeta + \frac{\tau}{2\beta}< 1.
\end{equation}
Therefore, $(\tau,\sigma)$ satisfies the condition in Theorem~\ref{teo:conve1} ensuring the convergence of FPDHF for $\varepsilon = \tau/(2\beta)$. Additionally, by setting $\zeta = 1$, $\beta = 1/2$, and $L$ such that $\|L\|=1$, \eqref{eq:stepsize25} forces $\tau < 1/3$, but FPDHF admits, for example, $\tau = 1/2$ if $\sigma < 1/2$, showing that PDBTR is more restrictive on the step-sizes than FPDHF. On the other hand, from \cite[Corollary~6.3]{MorinBanertGiselsson2022}, the convergence of PDRCK is guaranteed for $(\tau,\sigma)$ such that 
\begin{equation}\label{eq:stepsize3}
2\tau \sigma \|L\|^2 +2\tau\zeta + \frac{\tau}{2\beta} < 1.
\end{equation}
By setting $\lambda_n\equiv 2$ in \eqref{eq:stepsize2}, it is directly that PDRCK is more restrictive on the step-sizes than PDBTR, and consequently, than FPDHF. This advantage of FPDHF in allowing larger step sizes can translate into numerical benefits as demonstrated in the following example.
In the context of Section~\ref{sec:subsecNE1}, we consider the following parameters: $N \in \{128, 256\}$, $\lambda_1 \in \{10^{-1},10^{-2},10^{-3}\}$, $\lambda_2 = 10^{-3}$, Haar basis of level 3, and $\delta =10^{-4}$. In this setting, $\delta$ is smaller than in the previous section, resulting in a larger Lipschitz constant, which in turn makes $\tau$ smaller. The step-sizes are detailed in Table~\ref{T:stepsize2}. $\tau$ is chosen to be as large as possible by a factor of $0.95$ and $\sigma$ is determined according to Table~\ref{T:stepsize}. Due the difficulty of setting the step-sizes for FBHFPD (see \cite[Corollary~5.2]{BricenoDavis2018}), we select $\tau$ as in FPDHF. 
Although this value is admissible for FBHFPD, the corresponding value of $\sigma$ for FBHFPD is smaller than that for FPDHF. The results are shown in Table~\ref{T:128_delta}, Figure~\ref{fig:graphcomparison_4}, and  Figure~\ref{fig:graphcomparison_5}. We can note that FPDHF has the best performance in terms of iteration number and CPU time. The advantages are even pronounced when $\lambda_1$ decreases. This can be explained by the fact that $\lambda_1$ is related to the dual variable; hence, as $\lambda_1$ decreases, the primal term become more dominant in the optimization problem, allowing a larger $\tau$ to provide numerical benefits.
\begin{center}
\begin{table}\centering \resizebox{12cm}{!}{
		\begin{tabular}{|c|c|c|}\hline
			Algorithm & $\tau$ & $\sigma$ \\ \hline 
			FPDHF &  $0.95\cdot(\sqrt{1/(4\beta^2)+4\zeta^2)}-1/(2\beta))/(2\zeta^2)\approx 0.0927$ & $\approx 0.1284$ \\
			\hline
			FBHFPD &  $0.95\cdot(\sqrt{1/(4\beta^2)+4\zeta^2)}-1/(2\beta))/(2\zeta^2)\approx 0.0927$ & $\approx 0.0398$\\
			\hline 
			PDBTR & $0.95\cdot1/(2\zeta+1/(2\beta))\approx 0.0463$ & $\approx 0.1349$\\
			\hline
			BDRCK &  $0.95\cdot1/(2\zeta+1/(2\beta))\approx 0.0463$ & $\approx0.0674$ \\
			\hline
	\end{tabular}}	\caption{Step-sizes implemented in the experiments. $\zeta = \lambda_2/\delta $ and $\beta = 1$.}\label{T:stepsize2}
\end{table}
\end{center}
\begin{center}
\begin{table}\centering \resizebox{12cm}{!}{
		\begin{tabular}{|c|c|c|c|c|c|c|c|c|c|}
			\cline{5-10}  \multicolumn{4}{ }{}&\multicolumn{3}{|c|}{Haar Basis Level 3 }&\multicolumn{3}{|c|}{$\delta = 10^{-4}$} \\
			\hline
			Algorithm&I.N. & T. (s) & O.V. & I.N. & T. (s) & O.V. & I.N. & T. (s) & O.V. \\ 
			\hline
			\hline
			$N=128$ &\multicolumn{3}{|c|}{$\lambda_1=10^{-1}$, $\lambda_2=10^{-3}$}&\multicolumn{3}{|c|}{$\lambda_1=10^{-2}$, $\lambda_2=10^{-3}$}&\multicolumn{3}{|c|}{$\lambda_1=10^{-3}$, $\lambda_2=10^{-3}$}\\ \hline
			FPDHF & 7083&29.6&28.24&7942&32.1&4.35&9095&37.9&1.48\\ 
			\hline
			FBHFPD & 14491&66.5&28.25&13289&58.3&4.35&14235&63.4&1.48\\ 
			\hline
			PDBTR & 8856&30.5&28.24&11167&38.8&4.35&12646&44.8&1.48\\ 
			\hline
			BDRCK &	10178&36.8&28.25&11114&40.1&4.35&12649&47.4&1.48\\
			\hline
			\hline
			$N=256$ &\multicolumn{3}{|c|}{$\lambda_1=10^{-1}$, $\lambda_2=10^{-3}$}&\multicolumn{3}{|c|}{$\lambda_1=10^{-2}$, $\lambda_2=10^{-3}$}&\multicolumn{3}{|c|}{$\lambda_1=10^{-3}$, $\lambda_2=10^{-3}$}\\ \hline
			FPDHF &12407&257.3&161.0&10269&204.7&22.86&10414&208.2&7.79\\ 
			\hline
			FBHFPD &17392&329.4&161.0&16023&321.6&22.86&16251&323.6&7.79\\ 
			\hline
			PDBTR &15001&263.1&161.0&13806&242.3&22.86&14460&253.24&7.79\\ 
			\hline
			BDRCK &15531&283.3&161.0&13816&251.9&22.86&14463&264.1&7.79 \\
			\hline
	\end{tabular}}\caption{Comparison of FPDHF, FBHFPD, PDBTR, and BDRCK.}\label{T:128_delta}
\end{table}
\end{center}
\begin{figure}
\subfloat[]{\label{fig:graph_31}
	\includegraphics[scale=0.322]{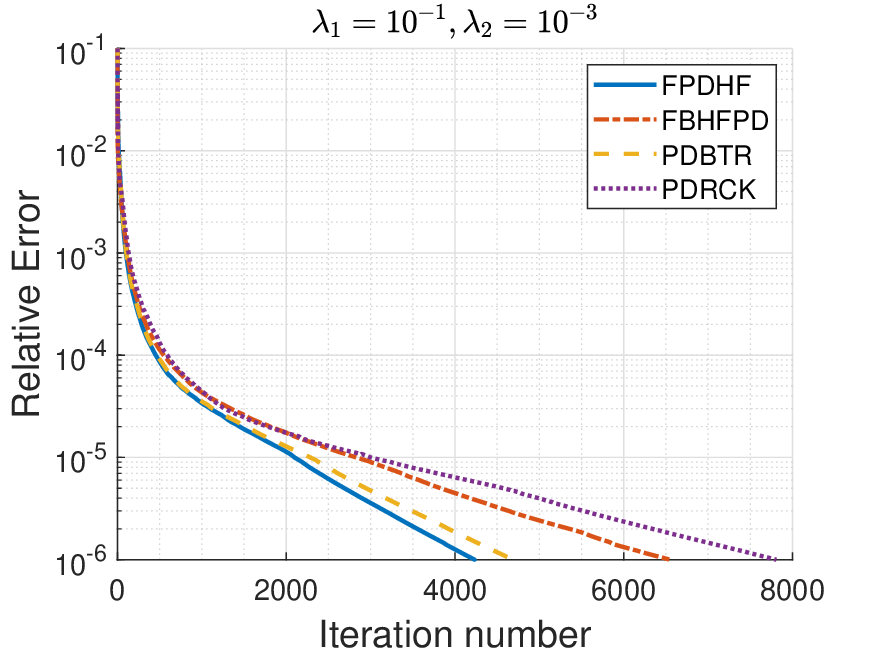}}
\subfloat[]{\label{fig:graph_32}
	\includegraphics[scale=0.322]{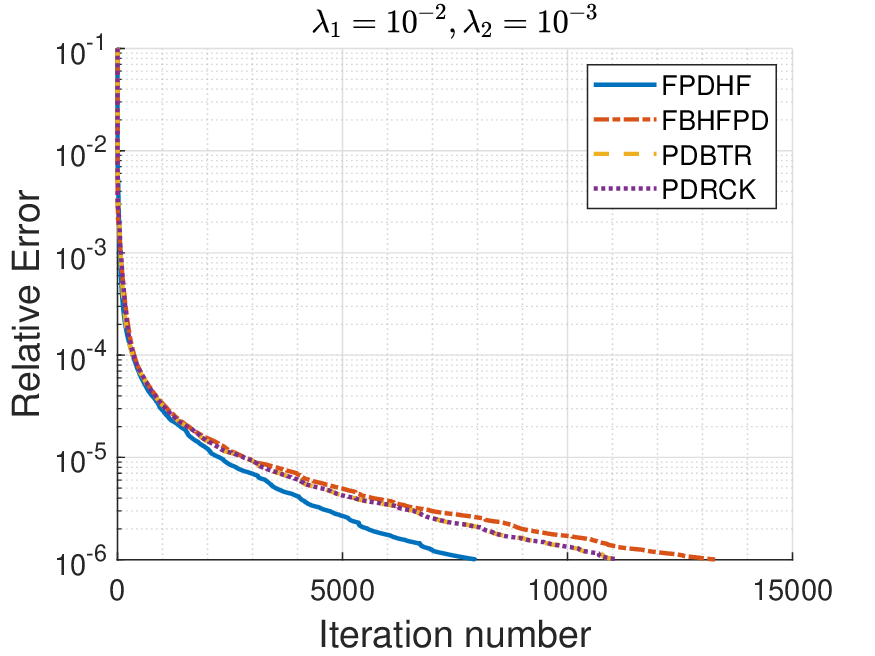}}
\subfloat[]{\label{fig:graph_33}
	\includegraphics[scale=0.322]{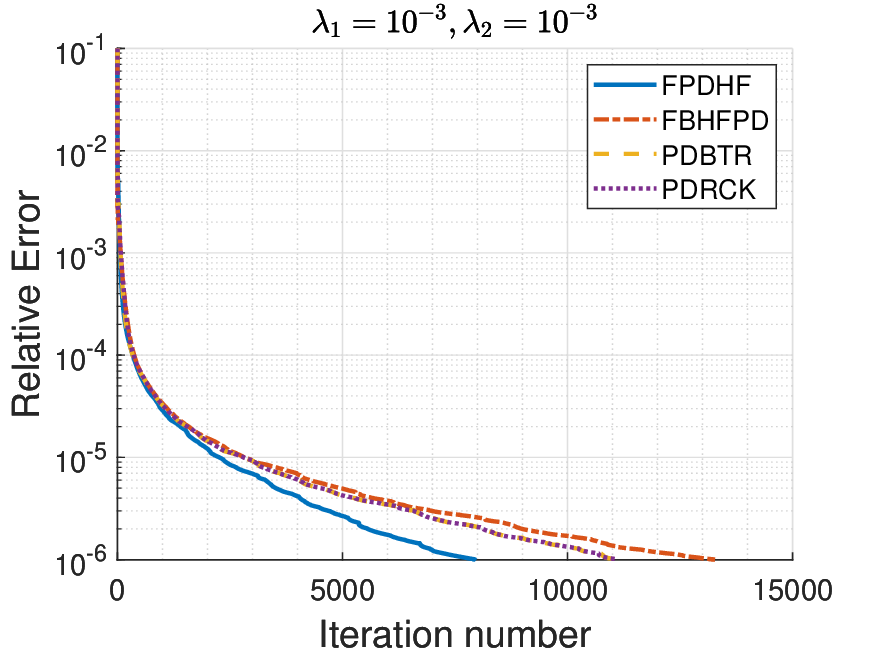}}
%	\subfloat[]{\label{fig:graph_13}
	%		\includegraphics[scale=0.325]{TIME_IT.eps}}
\caption{Relative error along iterations.  $N=128$, $\delta=10^{-3}$, Haar basis level 3.} \label{fig:graphcomparison_4}
\end{figure}
\begin{figure}
\subfloat[]{\label{fig:graph_41}
	\includegraphics[scale=0.322]{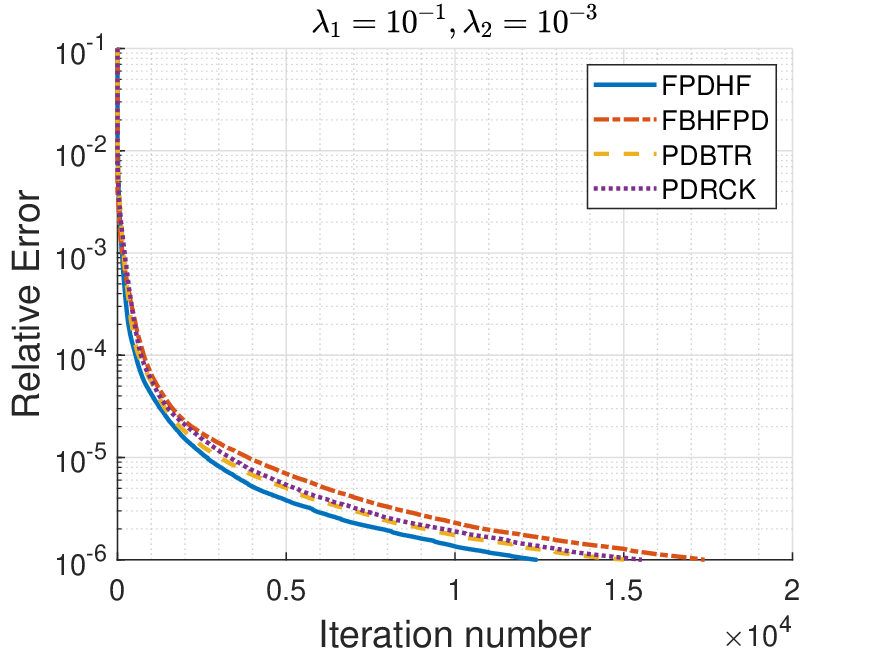}}
\subfloat[]{\label{fig:graph_42}
	\includegraphics[scale=0.322]{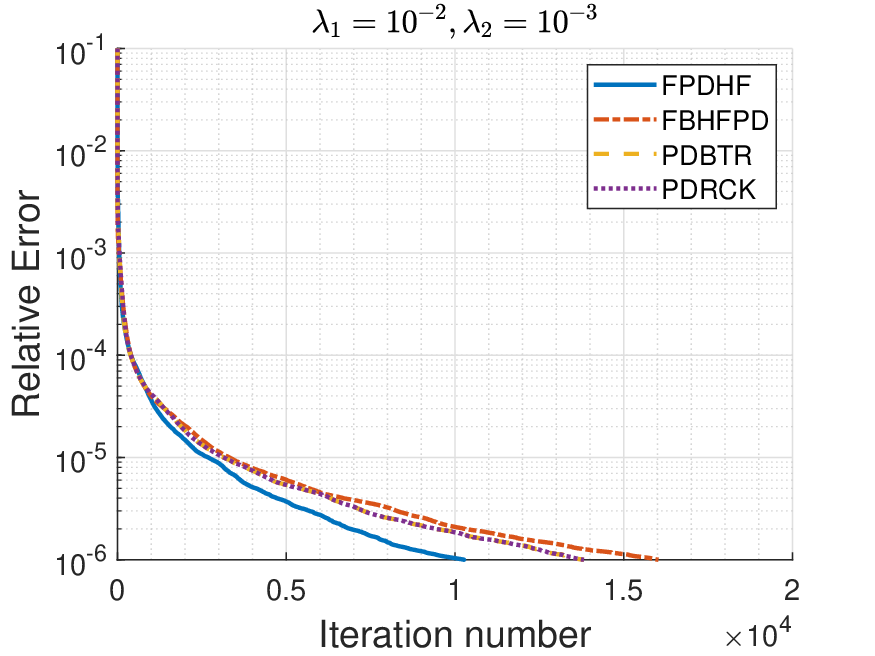}}
\subfloat[]{\label{fig:graph_43}
	\includegraphics[scale=0.322]{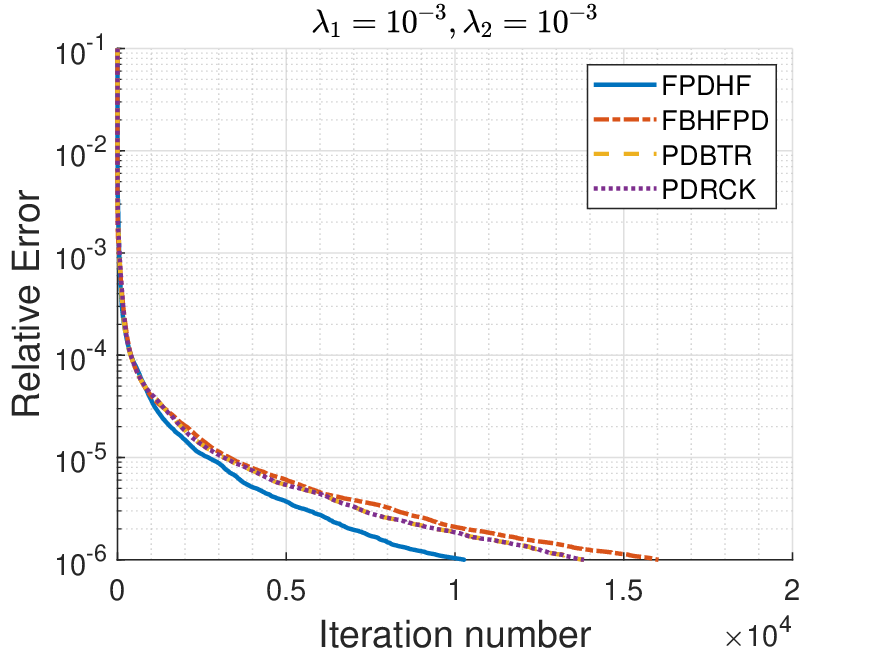}}
%	\subfloat[]{\label{fig:graph_13}
	%		\includegraphics[scale=0.325]{TIME_IT.eps}}
\caption{Relative error along iterations. $N=256$, $\delta=10^{-3}$, Haar basis level 3.} \label{fig:graphcomparison_5}
\end{figure}
\section{Conclusion}
In this article, we have introduced an algorithm to find zeros of the sum of a maximally monotone operator, a maximally monotone operator composed with a linear operator, a Lipschitz operator, and a cocoercive operator. On each iteration, our method splits all the operators involved in the problem. In specific scenarios where the Lipschitz operator is absent or when the linear composite term is absent, our method reduces to the Condat-V\~u algorithm or to the Forward-Backward-Half-Forward algorithm (FBHF), respectively, with their respective step-sizes. We have also presented applications in multivariate monotone inclusions and saddle point problems. To conclude, we have compared the proposed method with existing methods in the literature in image deblurring problems. {In this numerical example, we observed numerical advantages of the proposed method over existing methods in the literature. These advantages are mainly due the fact that the proposed method allows for larger step sizes.  Future work arising from this manuscript includes, for example, studying the backtracking versions of the algorithm (similar to the backtracking versions of FBF and FBHF \cite{BricenoDavis2018,Tseng2000SIAM}) and their convergence rates in primal-dual optimization problems.}

\section*{Acknowledgments}
The author acknowledges the support of Inria-Saclay during their postdoctoral studies.

%\bibliographystyle{spmpsci}
%\bibliography{ref+}	

\end{document}